\documentclass[11pt,reqno]{amsart}
\usepackage{amsmath,amssymb,amsthm,tikz-cd,mathrsfs,enumitem,mathtools}
\usepackage[math]{cellspace}
\usepackage{hyperref}
\usepackage[T1]{fontenc}

\numberwithin{equation}{section}

\newtheorem{thm}{Theorem}[section]
\newtheorem{cor}[thm]{Corollary}
\newtheorem{lem}[thm]{Lemma}
\newtheorem{prop}[thm]{Proposition}
\theoremstyle{definition}
\newtheorem{defn}[thm]{Definition}
\theoremstyle{remark}
\newtheorem{rmk}[thm]{Remark}
\newtheorem{exmp}[thm]{Example}

\newcommand{\N}{\mathbb{N}}
\newcommand{\Q}{\mathbb{Q}}

\newcommand{\GL}{\operatorname{GL}}

\newcommand{\SL}{\operatorname{SL}}

\newcommand{\SO}{\operatorname{SO}}

\newcommand{\tr}{\operatorname{tr}}

\newcommand{\ad}{\operatorname{ad}}

\renewcommand{\tr}{\operatorname{tr}}
\newcommand{\Tr}{\operatorname{Tr}}
\newcommand{\Trreg}{\operatorname{Tr}_{\operatorname{reg}}}

\newcommand{\vol}{\operatorname{vol}}
\renewcommand{\ker}{\operatorname{ker}}

\newcommand{\Hom}{\operatorname{Hom}}
\newcommand{\End}{\operatorname{End}}
\renewcommand{\dim}{\operatorname{dim}}
\renewcommand{\sp}{\operatorname{span}}
\newcommand{\rank}{\operatorname{rank}}

\newcommand{\unip}{\operatorname{unip}}
\newcommand{\geo}{\operatorname{geo}}
\newcommand{\spec}{\operatorname{spec}}
\newcommand{\Ind}{\operatorname{Ind}}

\newcommand{\FP}{\operatorname{FP}}

\newcommand{\mf}[1]{\mathfrak{#1}}

\newcommand{\A}{\mathbb{A}}
\newcommand{\C}{\mathbb{C}}
\newcommand{\Z}{\mathbb{Z}}
\newcommand{\R}{\mathbb{R}}

\renewcommand{\phi}{\varphi}
\renewcommand{\epsilon}{\varepsilon}
\renewcommand{\subset}{\subseteq}

\title{Asymptotics of analytic torsion for congruence quotients of $\SL(n,\R)/\SO(n)$}
\author{Tim Berland}
\address{Department of Mathematical Sciences, University of Copenhagen, Universitetsparken 5, 2100 Copenhagen, Denmark}
\email{twb@math.ku.dk}

\begin{document}
\maketitle

\markright{\MakeUppercase{Asymptotics of analytic torsion for quotients of} $\SL(n,\R)/\SO(n)$}

\noindent \textbf{Abstract.}
In this paper we prove a sharpened asymptotic for the growth of analytic torsion of congruence quotients of $\SL(n,\R)/\SO(n)$ in terms of the volume. The result is based on bounds on the trace of the heat kernel, allowing control of the large time behaviour of certain orbital integrals, as well as a careful analysis of error terms. The result requires the existence of $\lambda$-strongly acyclic representations, which we define and show exists in plenitude for any $\lambda>0$. The motivation is possible applications to torsion in the cohomology of arithmetic groups.

\tableofcontents

\pagebreak

\section{Introduction}\label{introduction}
\noindent Analytic torsion is an invariant of Riemannian manifolds, defined by Ray and Singer in the $70$'s (\cite{RS}). More recently, it was used by Bergeron and Venkatesh to study the torsion in the homology of cocompact arithmetic groups (\cite{BV}). This made it desirable to extend the definition of analytic torsion to manifolds associated to non-cocompact arithmetic groups. This was first accomplished by Müller--Pfaff (\cite{MP}) for finite volume hyperbolic manifolds. Later, it was extended further by Matz and Müller, first for congruence subgroups of $\SL(n)$ in \cite{MzM1}, and later for more general arithmetic groups in \cite{MzM3}. 

One of the reasons for studying torsion in the homology of arithmetic groups is the Ash conjecture (\cite{Ash}), partially proven by Scholze (\cite{Scholze}). In rough terms, this conjecture links systems of Hecke eigenvalues occurring in the mod $p$ cohomology of arithmetic groups to Galois representations with matching Frobenius eigenvalues. See (\cite{BV}, Section $6$) for a discussion on the Ash conjecture and its connection to analytic torsion. As explained in Scholze's paper, by applying the work of Bergeron--Venkatesh and others, in certain situations the Ash conjecture predicts the existence of many Galois representations, more than is predicted by the global Langlands correspondence. 

Further control of the behaviour of analytic torsion should translate to better understanding of torsion homology, and conjecturally to existence of Galois representations. With this motivation we give improved asymptotics of analytic torsion in terms of the volume of the associated manifolds. To state our theorem and the connection to homology, we recall the setups of \cite{BV} and \cite{MzM2}.

\subsection*{Analytic torsion in the compact setting}

Let $X$ be a compact oriented Riemannian manifold of dimension $d$, and let $\rho$ be a finite-dimensional representation of $\pi_1(X)$. To such a representation, one may associate a flat vector bundle $E_\rho\to X$ over $X$. Fixing a Hermitian fiber metric on $E_\rho$, we let $\Delta_p(\rho)$ be the Laplace operator on $E_\rho$-valued $p$-forms. For $t>0$, we define $e^{-t\Delta_p(\rho)}$ to be the heat operator, and we let $k_p(\rho) =\dim\ker \Delta_p(\rho)$. One can associate a zeta function $\zeta_p(s,\rho)$ to $\Delta_p(\rho)$, a meromorphic function on $\C$ defined by
\begin{align}\label{zeta}
    \zeta_p(s,\rho) = \frac{1}{\Gamma(s)}\int_0^\infty (\Tr e^{-t\Delta_p(\rho)}-k_p(\rho))t^{s-1} dt
\end{align}

\noindent for $\Re(s)>\frac{d}{2}$. We then define analytic torsion $T_X(\rho)$ of $M$ by
\begin{align}
    \log T_X(\rho) =\frac12 \sum_{p=1}^d (-1)^p p\frac{d}{ds}\zeta_p(s;\rho)\vert_{s=0}.
\end{align}

\noindent There is a corresponding $L^2$-invariant, called $L^2$-torsion denoted $T^{(2)}_X(\rho)$, defined by Lott and Mathai (resp. \cite{Lott}, \cite{Mathai}). For $\Tilde{X}$ the universal covering of $X$, it may be expressed as
\begin{align}\label{l2torsion}
    \log T^{(2)}_X(\rho) = t^{(2)}_{\Tilde{X}}(\rho)\vol(X),
\end{align}
where $t^{(2)}_{\Tilde{X}}(\tau)$ is an invariant associated to $\Tilde{X}$ and $\rho$.

To relate this to arithmetic groups, let $G$ be a semisimple algebraic group with $G(\R)$ of non-compact type, let $K$ be a maximal compact subgroup of $G(\R)$, and suppose $\Gamma\subset G(\Q)$ is a torsion-free congruence lattice. Assume further that $\Gamma$ is cocompact in $G(\R)$. Then $\Tilde{X}\coloneqq G(\R)/K$ is a Riemannian symmetric space, and $X\coloneqq \Gamma\backslash \Tilde{X}$ is a compact locally symmetric space, and as $\Tilde{X}$ is contractible, we have $\pi_1(X)=\Gamma$. Thus, given $\Gamma$ and a finite-dimensional representation $\rho$ of $\Gamma$, we associate to it the analytic torsion $T_X(\rho)$ of the space $X$ as defined above. 

By a theorem of Cheeger and Müller, later generalized by Bismut--Zhang (\cite{Cheeger}, \cite{Müller1}, \cite{Müller2}, \cite{BZ}), there is an equality of analytic torsion and \textit{Reidemeister torsion} on $X$, the latter being a topological invariant of $X$ computed in terms of homology (see \cite{Reidemeister}, \cite{Franz}). Since $X$ is a classifying space for $\Gamma$, the homology of $X$ computes the homology of $\Gamma$, and thus one may study the homology of $\Gamma$ through its associated analytic torsion.

Consider a descending chain of finite index subgroups $\Gamma=\Gamma_0\supseteq\Gamma_1\supseteq...$ in $G(\Q)$ satisfying $\bigcap_{i=1}^\infty\Gamma_i = \lbrace1\rbrace$, and let $X_i\coloneqq \Gamma_i\backslash \Tilde{X}$. Let $\tau$ be an irreducible finite-dimensional representation of $G(\R)$. By abuse of notation, we set $T_{X_i}(\tau)\coloneqq T_{X_i}(\tau|_{\Gamma_i})$. Given a non-degeneracy condition on $\tau$, namely requiring that it be \textit{strongly acyclic}, Bergeron and Venkatesh prove (\cite{BV}, Theorem $4.5$) that
\begin{align}\label{BVapprox}
    \lim_{i\to\infty} \frac{\log T_{X_i}(\tau)}{\vol(X_i)} = t^{(2)}_{\Tilde{X}}(\tau).
\end{align}

\noindent Bergeron and Venkatesh show that $t^{(2)}_{\Tilde{X}}(\tau)$ is non-zero if and only if the \textit{deficiency} of $G$, defined $\delta(G)\coloneqq\rank G(\R)-\rank K$, is $1$ (\cite{BV}, Proposition $5.2$). The assumption that $\tau$ is strongly acyclic also guarantees that the free homology is trivial, so Reidemeister torsion is expressed solely in torsion homology. Using the Cheeger-Müller theorem, this is interpreted as there being a lot of torsion in the homology when $\delta(G)=1$, and little torsion in all other cases. "Little torsion" should here be understood in a weak sense, as we only know that there is not full exponential growth of analytic torsion in terms of the volume. A natural question is then, also presented in \cite{AGMY}: what growth rate should we expect when the deficiency is not $1$, and can we more precisely describe the behaviour of analytic torsion in terms of the volume? The results of this paper is a first step in answering these questions.

The result (\ref{BVapprox}) can be thought of as an approximation theorem. In particular, noting that every $\Gamma_i$ is a finite index subgroup of $\Gamma$, we see that $\vol(X_i)=\vol(X)[\Gamma:\Gamma_i]$, where $X\coloneqq X_0$, and thus the result is equivalent to 
\begin{align}
    \lim_{i\to\infty} \frac{\log T_{X_i}(\tau)}{[\Gamma:\Gamma_i]} = \log T^{(2)}_{X}(\tau).
\end{align}

\noindent In words, analytic torsion associated to a tower of subgroups can be used to approximate the $L^2$-torsion associated to the group. Keeping in mind the equality of analytic torsion and Reidemeister torsion, this statement should be thought of as a torsion version of the analogous result on Betti numbers $b_p(X)$ and $L^2$-Betti numbers $b_p^{(2)}(X)$, proven by W. Lück (\cite{Lück1}), which we state here in our setting.
\begin{align}
    \lim_{i\to\infty} \frac{b_p(X_i)}{[\Gamma:\Gamma_i]} = b_p^{(2)}(X).
\end{align}

\noindent To further the analogy, here there is a criterion based on the deficiency as well: We have that $b_p^{(2)}(X)=0$ unless $\delta(G) = 0$. For a comprehensive survey on approximation of $L^2$-invariants, see \cite{Lück2}.

For certain approximation theorems, their rate of convergence has been explored (see e.g. \cite{BLLS}, and \cite{Lück2} Chapter 5), but is not very developed for analytic and $L^2$-torsion. This is mostly due to the fact that this approximation result is still conjectural in general, and only proven in very particular cases, such as the setting of this paper. Viewed in this light, Theorem \ref{mytheorem} presented below can be seen as an example of a rate of convergence result for analytic torsion in this aspect.

Many of the most important arithmetic groups in number theory are not cocompact. In general for non-cocompact arithmetic groups in semisimple Lie groups, we do not have a replacement for the Cheeger-Müller theorem, though certain special cases have recently been proven (see \cite{MR1}, \cite{MR3}). The non-cocompact setting do, however, have one  advantage, as it allows certain terms to contribute (namely the non-identity unipotent part, see section \ref{applyingfine}) which vanish in the cocompact setting, and these terms are very probable candidates for where second order terms should show up. For this reason, the present paper will focus on the non-cocompact setting. 

\subsection*{Analytic torsion in the non-compact setting} 

To state our main theorem, let us recall the setup of Matz and Müller (\cite{MzM1}, \cite{MzM2}). We switch to an adelic framework and focus our attention on $\SL(n)$, $n\geq 2$. Let $\A$ be the ring of adeles of $\Q$, with $\A_f$ the finite adeles. Let $X=\SL(n,\R)/\SO(n)$. For $N\geq 3$, let $K(N)\subset \SL(n,\A_f)$ be the open compact subgroup given by $K(N)=\prod_p K_p(p^{\nu_p(N)})$, with $K_p(p^e)\coloneqq \ker(\SL(n,\Z_p)\to \SL(n,\Z_p/p^e\Z_p))$. Define
\begin{align*}
    X(N) \coloneqq \SL(n,\Q)\backslash (\Tilde{X}\times \SL(n,\A_f))/K(N)).
\end{align*}

\noindent By strong approximation, we get that $X(N) = \Gamma(N)\backslash \Tilde{X}$, with $\Gamma(N)\subset \SL(n,\Z)$ the standard principal congruence subgroup of level $N$. Given a finite-dimensional representation $\tau$ of $\SL(n,\R)$, one can now follow the construction of the Laplace operator $\Delta_p(\tau)$ of the Laplace operator on $p$-forms with values in $E_\tau$, as above. As $X(N)$ is not compact, however, $\Delta_p(\tau)$ has continuous spectrum, and hence we need a refined definition of the analytic torsion. This is constructed by Matz and Müller in \cite{MzM1}. Most strikingly, they define the regularized trace of the heat kernel $\Trreg(e^{-t\Delta_p(\tau)})$ as the geometric side of the \textit{Arthur trace formula} for $\SL(n,\R)$ applied to the test function $h_t^{\tau,p}\otimes \chi_{K(N)}$, where $h_t^{\tau,p}$ is the trace of the heat kernel and $\chi_{K(N)}$ the normalized characteristic function of $K(N)$. We give the full definition in 
Section \ref{torsion}. This also reduces to the standard definition in the cocompact setting. In \cite{MzM2}, they prove the same limit behaviour as Bergeron-Venkatesh for this setup, namely the approximation formula
\begin{align}\label{MzM2mainresult}
    \lim_{N\to\infty}\frac{\log T_{X(N)}(\tau)}{\vol(X(N))} = t^{(2)}_{\Tilde{X}}(\tau).
\end{align}

\noindent Formulated in terms of asymptotics of analytic torsion, the above is equivalent to
\begin{align}\label{MzMasymp}
    \log T_{X(N)}(\tau)=\log T^{(2)}_{X(N)}(\tau)+o(\vol(X(N))) \quad\text{as}\quad N\to\infty.
\end{align}

\subsection*{Results} 

As a step towards better understanding the behavior of analytic torsion in terms of the volume, in this paper we prove a stronger version of the asymptotic (\ref{MzMasymp}). In particular, the new result provides an improved upper bound on analytic torsion when the deficiency is not $1$, as well as a bound on second order terms when the deficiency is $1$. We need a slightly stronger non-degeneracy assumption on $\tau$ which we call $\lambda$-strongly acyclic. Our result is the following.
\begin{thm}\label{mytheorem}
    Assume $\tau$ is a $\lambda$-strongly acyclic representation of $\SL(n,\R)$, for a certain $\lambda$ depending only on $n$. Then there exists some $a>0$ such that
    $$\log T_{X(N)}(\tau) = \log T_{X(N)}^{(2)}(\tau)+O(\vol(X(N))N^{-(n-1)}\log(N)^a)$$
    as $N$ tends to infinity.
\end{thm}

\begin{rmk}\label{mytheoreminvar}
    Computing the size of $\vol(X(N))$ in terms of $N$ (see the appendix), the theorem implies a more intrinsic version:
\begin{align}
    \log T_{X(N)}(\tau) = \log T_{X(N)}^{(2)}(\tau)+O(\vol(X(N))^{1-\frac{1}{n+1}}\log(\vol(X(N))^a)
\end{align}
\noindent as $N\to\infty$.
\end{rmk}

\noindent To ensure that the technical constraint is justified, we prove the existence of infinitely many $\lambda$-strongly acyclic representations for any connected semisimple algebraic group $G$ with $\delta(G)\geq 1$. This result is an extension of the result (\cite{BV}, Section $8.1$), stating that strongly acyclic representations always exist for semisimple $G$ with $\delta(G)=1$.

It seems likely that up to log-terms, the bound in the theorem is strict in the case of deficiency $1$, though this is not proven here. When the deficiency is not $1$, a better upper bound is expected. A lower bound would be extremely interesting as well, but this is not in the scope of this paper.

Because of technical reasons, we will work with $\GL(n)$ instead of $\SL(n)$. For $Y(N)$ the analogous adelic locally symmetric space of $\GL(n)$, one can define analytic torsion $T_{Y(N)}(\tau)$ in a similar manner (see (\ref{torsion}) for the general definition). Our proof is then that of Theorem \ref{glmytheorem}, which is the analogous statement of Theorem \ref{mytheorem} for $Y(N)$, and we show the two theorems are equivalent.

The proof is based on the framework of \cite{MzM2}. The main contribution of this paper lies in the change of perspective to allow certain parameters, namely the compactification parameter (see Section \ref{compactsub}) and the truncation parameter $T$ (see (\ref{trunc})) to vary with the level $N$, and in the work needed to allow for this variation. In particular, we show the existence of infinitely many $\lambda$-strongly acyclic representations of any connected semisimple algebraic group (Proposition \ref{kstrong}), and we adapt to our setting certain bounds on the trace of the heat kernel in the large time aspect (e.g. Proposition \ref{traceheatdecay}). These bounds are used to gain control over large time behaviour of the archimedean orbital integrals showing up in the analysis of the Arthur trace formula.

\subsection*{Organization of the paper} 

In Section \ref{definitions}, we present the setup for the heat kernel on symmetric spaces. In Section \ref{representations}, we prove the existence of infinitely many $\lambda$-strongly acyclic representations for any semisimple algebraic group. Section \ref{plancherel} gives a proof of large $t$ asymptotics of the trace of the heat kernel. We give a brief introduction to the Arthur trace formula and its geometric expansions in Section \ref{trace formula} and define analytic torsion, as well as express the trace formula applied to a compactification of our test function. In Section \ref{asymptotics}, we handle the asymptotics of our local orbital integrals. Finally, we combine all the ingredients in Section \ref{conclusion} to prove our main theorem.

\subsection*{Aknowledgement}

The author would like to thank Jasmin Matz, for suggesting this topic and for her advice throughout the project. I thank the University of Bonn for their hospitality while working on this project. I would also like to thank Werner Müller, for our helpful discussions while in Bonn and for his comments to a draft of this paper.

\medskip

\noindent The work presented here is supported by the Carlsberg Foundation, grant CF21-0374.
\bigskip
\bigskip

\section{Preliminaries}\label{definitions}
\subsection*{Arithmetic manifolds}\label{arithmfd}

This section will deal with the general setup of $G$ being a reductive algebraic group defined over $\Q$ with $\Q$-split zenter $Z_G$. See (\cite{Arthur0}, Section I.$2$) for a partial reference.  Let $K_f\subset G(\A_f)$ be any open compact subgroup. We write $G(\A) = G(\A)^1\times A_G$ for $A_G$ the identity component of $Z_G$, and we set $G(\R)^1\coloneqq G(\A)^1\cap G(\R)$, which is then a semisimple real Lie group. We fix a Cartan involution $\theta$ of $G(\R)$ and let $K$ denote its fixpoints. We set $\Tilde{X}\coloneqq G(\R)^1/K$. The group $G(\R)$ has a natural action on the finite double coset space,

\begin{align*}
    A_G(\R)^0G(\Q)\backslash G(\A)/G(\R)K_f.
\end{align*}
Taking a set of representatives $z_1,\dots,z_m$ for the double cosets in $G(\A_f)$ and defining 
\begin{align*}
    \Gamma_j \coloneqq (G(\R)\times z_jK_fz_j^{-1})\cap G(\Q)
\end{align*}
for all $1\leq j\leq m$, the action induces the decomposition
\begin{align}\label{decomparithm}
    A_G(\R)^0 G(\Q)\backslash G(\A)/K_f \cong \bigsqcup_{j=1}^m (\Gamma_j\backslash G(\R)^1).
\end{align}
Now we define the arithmetic manifold associated to $K_f$. Set
\begin{align}\label{adelicsym}
    X(K_f)\coloneqq G(\Q)\backslash (\Tilde{X} \times G(\A_f))/K_f.
\end{align}
Using (\ref{decomparithm}), we get
\begin{align}\label{decompsym}
    X(K_f)\cong \bigsqcup_{j=1}^m (\Gamma_j\backslash \Tilde{X}).
\end{align}
Here, $\Gamma_j\backslash \Tilde{X}$ is a locally symmetric space. We will assume that $K_f$ is neat such that $X(K_f)$ is a locally symmetric manifold of finite volume.

\subsection*{The heat kernel}\label{heatker}

\noindent In the setup above, given a finite-dimensional unitary representation of $K$, we let $\Tilde{E}_\nu$ be the associated homogeneous Hermitian vector bundle over $\Tilde{X}$. Using homogeneity, one can push down this vector bundle to give locally homogeneous Hermitian vector bundles over each component $\Gamma_j\backslash \Tilde{X}$. Taking the disjoint union of these vector bundles, we get a locally homogeneous vector bundle $E_\nu$ over $X(K_f)$. 

As it is sufficient to work on each component in (\ref{decompsym}), we continue this section with the simplifying assumption that $G$ is a connected semisimple algebraic group. Let $K\subset G(\R)$ be a maximal compact subgroup, and $\Gamma\subset G(\R)$ a torsion-free lattice. We let $\Tilde{X}=G(\R)/K$ and $X=\Gamma\backslash \Tilde{X}$. Take a finite-dimensional unitary representation $(\nu,V_\nu)$ of $K$ with inner product $\langle\cdot,\cdot\rangle_\nu$, and define
\begin{align*}
    \Tilde{E_\nu} \coloneqq G(\R)\times_\nu V_\nu
\end{align*}
as the associated homogenous vector bundle over $\Tilde{X}$. The action of $K$ on $G(\R)$ is by right multiplication. The inner product $\langle\cdot,\cdot\rangle_\nu$ induces a $G(\R)$-invariant metric $\Tilde{h}_\nu$ on this vector bundle. 
Let $E_\nu\coloneqq \Gamma\backslash\Tilde{E_\nu}$ be the associated locally homogenous vector bundle over $X$, with metric $h_\nu$ induced by $\Tilde{h}_\nu$ using $G(\R)$-invariance. 
Let $C^\infty(\Tilde{X},\Tilde{E}_\nu)$ denote the space of smooth sections of $\Tilde{E}_\nu$. Now, set
\begin{align*}
    C^\infty(G(\R),\nu) \coloneqq \lbrace &f:G(\R)\to V_\nu \mid f\in C^\infty,\\
    &f(gk) = \nu(k)^{-1}f(g) \:\forall k\in K,g\in G(\R)\rbrace.
\end{align*}
There is an isomorphism (\cite{Miatello}, p. 4) interpreting the smooth sections as smooth functions on $G(\R)$,
\begin{align}\label{c-inf spaces}
    C^\infty(\Tilde{X},\Tilde{E}_\nu) \xrightarrow{\sim} C^\infty(G(\R),\nu).
\end{align}
\noindent This extends to an isometry of corresponding $L^2$-spaces. Let $\mathcal{C}(G(\R))$ denote Harish-Chandra's Schwartz space, and $\mathcal{C}^q(G(\R))$ Harish-Chandra's $L^q$-Schwartz space. 

We now specify the above to our setting. Let $(\tau,V_\tau)$ be an irreducible finite-dimensional representation of $G(\R)$, $E_\tau\coloneqq E_{\tau|_K}$, and $F_\tau$ the flat vector bundle over $X$ associated to the restriction of $\tau$ to $\Gamma$. Then we have a canonical isomorphism (see \cite{MM}, Proposition $3.1$)
\begin{align*}
    E_\tau\cong F_\tau.
\end{align*}
As $K$ is compact there exists an inner product on $V_\tau$ with respect to which $\tau|_K$ is unitary. Fix such an inner product. From this we induce a metric on $E_\tau$, and hence on $F_\tau$ as well. Define $\Lambda^p(X,F_\tau)\coloneqq \Lambda^pT^*(X)\otimes F_\tau$. Under the isomorphism above, this is isomorphic to the vector bundle associated to the representation $\Lambda^p\text{Ad}^*\otimes \tau$ of $K$ on $\Lambda^p\mf{p}^*\otimes V_\tau$. We will denote this representation by $\nu_{\tau,p}$. 

Let $\Delta_p(\tau)$ be the Laplace operator on $\Lambda^p(X,F_\tau)$, and let $\Tilde{\Delta}_p(\tau)$ be its lift to the universal covering $\Tilde{X}$. Let also $\Tilde{F}_\tau$ be the pullback of $F_\tau$ to $\Tilde{X}$. Then $\Tilde{\Delta}_p(\tau)$ is an operator on the space $\Lambda^p(\Tilde{X},\Tilde{F}_\tau)$ of $\Tilde{F}_\tau$-valued $p$-forms on $\Tilde{X}$. By (\ref{c-inf spaces}), we get an isomorphism
\begin{align}\label{pformsrep}
    \Lambda^p(\Tilde{X},\Tilde{F}_\tau)\cong C^\infty(G(\R),\nu_{\tau,p}).
\end{align}
Let $R$ be the right regular representation of $G(\R)$ on $C^\infty(G(\R),\nu_{\tau,p})$, and $\Omega$ the Casimir element of $G(\R)$. With respect to the above isomorphism, Kuga's lemma implies
\begin{align}\label{casimirglobal}
    \Tilde{\Delta}_p(\tau)=\tau(\Omega)-R(\Omega).
\end{align}
The operator $\Tilde{\Delta}_p(\tau)$ is formally self-adjoint and non-negative. Regarded as an operator with domain the space of compactly supported smooth $p$-forms, it has a unique self-adjoint extension to $L^2(\Tilde{X},\Tilde{F}_\tau)$, the $L^2$-sections of $\Tilde{F}_\tau$, which we by abuse of notation will also denote $\Tilde{\Delta}_p(\tau)$. This extension inherits non-negativity. We denote by $e^{-t\Tilde{\Delta}_p(\tau)}$, with $t>0$, the heat semigroup associated to $\Tilde{\Delta}_p(\tau)$. Considered as a bounded operator on $L^2(G(\R),\nu_{\tau,p})$ under the extension of the isomorphism (\ref{pformsrep}), it is a convolution operator, and thus it has a kernel
\begin{align*}
    H_t^{\tau,p}:G(\R)\to \End(\Lambda^p\mf{p}^*\otimes V_\tau)
\end{align*}
called the heat kernel. It satisfies the following covariance property
\begin{align}\label{covar}
    H_t^{\tau,p}(k^{-1}gk')=\nu_{\tau,p}(k)^{-1}\circ H_t^{\tau,p}(g)\circ \nu_{\tau,p}(k'), \quad \forall k,k'\in K, g\in G(\R).
\end{align}
By analogy of the proof of (\cite{BM}, Proposition $2.4$), we have that 
\begin{align}\label{heatkernelspace}
    H_t^{\tau,p}\in \mathcal{C}^q(G(\R))\otimes \End(\Lambda^p\mf{p}^*\otimes V_\tau)
\end{align} 
for any $q>0$. We may define
\begin{align}\label{trheat}
    h_t^{\tau,p}(g)\coloneqq \tr H_t^{\tau,p}(g), \quad g\in G(\R),
\end{align}
for $\tr$ being the trace over $\Lambda^p\mf{p}^*\otimes V_\tau$. By (\ref{covar}) and (\ref{heatkernelspace}), we have that $h_t^{\tau,p}(g)\in \mathcal{C}^q(G(\R))_{K\times K}$ for any $q>0$, with $\mathcal{C}^q(G(\R))_{K\times K}$ the subspace of $\mathcal{C}^q(G(\R))$ consisting of left and right $K$-finite functions.

\bigskip

\section{Strongly acyclic representations}\label{representations}
\noindent We continue with the setup from Section \ref{heatker}. In particular, $G$ is a connected semisimple algebraic group. In the setting that $\Gamma\backslash\Tilde{X}$ is compact, (\cite{BV}, §$4$) defines an irreducible finite-dimensional representation $(\tau,V_\tau)$ of $G(\R)$ to be \textit{strongly acyclic} if there exists some positive constant $\eta>0$ such that every eigenvalue of $\Delta_p(\tau)$ is $\geq \eta$ for every choice of $\Gamma$ and every $p$. Furthermore, they show that for any such $\tau$ not fixed by the Cartan involution, then $\tau$ must be strongly acyclic. This condition, i.e. $\tau\neq \tau\circ\theta$, is shown to give the necessary bounds in the noncompact setting in (\cite{MzM2}, Lemma $6.1$), and is used throughout that paper. 

For our purposes, we need a slightly stronger condition that we define now. This will be the central object of this section.
\begin{defn}
    Let $(\tau,V_\tau)$ be a finite-dimensional representation of $G(\R)$, and let $\lambda>0$. We say that $\tau$ is $\lambda$\textit{-strongly acyclic} if
    \begin{align*}
        \tau(\Omega)-\pi(\Omega) \geq \lambda
    \end{align*}
    for all irreducible unitary representation $\pi$ satisfying $\Hom_K(\Lambda^p\mf{p}\otimes V_\tau^*,\pi)\neq 0$ for some $p$.
\end{defn}
\noindent We see the analogy to the definition of strongly acyclic by comparing to (\ref{casimirglobal}). For $\tau$ a $\lambda$-strongly acyclic representation, we will sometimes say that $\lambda$ is the \textit{spectral gap} of $\tau$. We now prove that there are plenty of such representations to choose from.

\begin{prop}\label{kstrong}
    For any $\lambda>0$, there exists infinitely many $\lambda$-strongly acyclic representations of $G(\R)$ if $\delta(G(\R))\geq 1$.
\end{prop}

\noindent The proof takes up the remainder of the section. We set $F=V_\tau^*$ to be the dual of the representation space to match the notation of \cite{BV}. Consider $\mathfrak{g}\coloneqq \text{Lie}(G(\R))$ with $\mf{g}_\C$ its complexification. Recall that $\theta$ is a Cartan involution on $\mf{g}$ with Cartan decomposition $\mf{g}=\mf{k}\oplus\mf{p}$. This also gives the decomposition $\mf{g}_\C=\mf{k}_\C\oplus\mf{p}_\C$. Let $\mf{h}^+$ be a Cartan subalgebra in $\mf{k}$. Denote by $\mf{h}$ the centralizer of $\mf{h}^+$ in $\mf{g}$, which is then a Cartan subalgebra in $\mf{g}$.

The following two lemmas are elementary and certainly well known, but we were not able to find precise references. Hence, we will give the proofs. Similar, but not identical, results are presented in (\cite{Helgason2}, Section VI.$3$).

\begin{lem}\label{decomp}
    There exists an abelian subalgebra $\mf{h}^-\subset \mf{p}$ such that $\mf{h}=\mf{h}^+\oplus \mf{h}^-$.
\end{lem}

\begin{proof}
    Take any basis $(K_1,\dots,K_d)$ of $\mf{h}^+$, and extend it to a basis 
    $$(K_1,\dots, K_d,X_1,\dots,X_s)$$
    of $\mf{h}$. Using the Cartan decomposition, write $X_i = M_i+P_i$ for $M_i\in \mf{k}$ and $P_i\in \mf{p}$. Since $[K,K']\in \mf{k}$ and $[P,K]\in \mf{p}$ for any $\lambda,K'\in\mf{k}$ and $P\in\mf{p}$, and $\mf{k}\cap \mf{p}=0$, we see that $[M_i,K_j]+[P_i,K_j] = [X_i,K_j]=0$ for all $i,j$ implies that $[M_i,K_j]=[P_i,K_j]=0$. In particular, $M_i\in \mf{k}$ commutes with all of $\mf{h}^+$, but as this is a \textit{maximal} abelian subalgebra of $\mf{k}$, we get $M_i\in \mf{h}^+$. Thus, it becomes clear that 
    $$\mf{h} = \text{span}\lbrace K_1,\dots, K_d,P_1,\dots,P_s\rbrace.$$
\end{proof}

\noindent Let $\Phi_k$ be the roots of $\mf{h}^+_\C$ in $\mf{k}_\C$, and let $\Phi$ be the roots of $\mf{h}_\C$ in $\mf{g}_\C$. Pick $\Phi^+_k$ a set of positive roots of $\Phi_k$, and let $\Phi^+$ be a \textit{compatible} choice of positive roots of $\Phi$ (see \cite{BW}, II, §$6.6$). A compatible root system $\Phi^+_k$ is defined as a root system which is closed under the Cartan involution (acting by precomposition), and every element in $\Phi^+_k$ is the restriction of some element in $\Phi^+$ to $\mf{h}^+_\C$. Let $\rho$ and $\rho_k$ be the respective half-sum of positive roots. Set $W$ to be the Weyl group of $\mf{g}_\C$, and let $W^1$ be the subset of elements $w$ such that $w\Phi^+$ is again a compatible system of positive roots.

We will be concerned with the weight lattice of $\mf{g}_\C$, which lives inside the real span of the roots, $\sp_\R \Phi$. To utilize our decomposition above, we need the following elementary lemma. We extend $(\mf{h}^+)^*$ to a subspace of $\mf{h}^*$ be setting elements to be identically zero on $\mf{h}^-$, and extend $(\mf{h}^-)^*$ by setting elements identically zero on $\mf{h}^+$. We further identify these as (real) subspaces of $\mf{h}^*_\C$ using that $\mf{h}^*_\C= \mf{h}^*\oplus i\mf{h}^*$. 

\begin{lem}\label{dualdecomp}
    We have a decomposition $\sp_\R \Phi = i(\mf{h}^+)^*\oplus (\mf{h}^-)^*$, and this decomposition is orthogonal with respect to the inner product induced by the Killing form.
\end{lem}

\begin{proof}
    Define $\mf{h}_0^* \coloneqq \sp_\R \Phi$. Given $H\in \mf{h}$, note that the values $\phi(H)$ for $\phi\in\Phi\sqcup \lbrace 0 \rbrace$ are exactly the eigenvalues of $\ad_H$, by definition. Also, it is an elementary fact that for any element $X$ in a compact subalgebra, all eigenvalues of $\ad_X$ are imaginary -- one could argue as follows: Its Lie group is compact, thus its action on the Lie algebra by conjugation has eigenvalues of absolute value $1$. Now take logarithms. 

As $\mf{k}$ is a compact Lie subalgebra of $\mf{g}_\C$, the above implies that for $K\in\mf{h}^+$ and $\phi\in\Phi_k$ we have $\phi(K)\in i\R$. Thus, the roots must take values in $\R$ on $i\mf{h}^+$. Similarly, as $i\mf{p}$ is a subspace of the compact real form $\mf{u}=\mf{k}\oplus i\mf{p}$ of $\mf{g}_\C$, the roots must take values in $i\R$ on $i\mf{p}$, and thus values in $\R$ on $\mf{p}$. This proves the inclusion $\mf{h}_0^* \subset i(\mf{h}^+)^*\oplus (\mf{h}^-)^*$, and equality follows from comparing dimensions.

The second claim follows easily from the fact that $\mf{k}\oplus \mf{p}$ is orthogonal with respect to the Killing form, hence so is $i\mf{h}^+\oplus \mf{h}^-$, and this space is isomorphic to its dual induced by the inner product given by the Killing form.

\end{proof}

\noindent As this may be a nonstandard choice of positive root systems to the reader, let us consider an example.

\begin{exmp}\label{SL(3,R)example}
    Consider $G(\R) = \SL(3,\R)$ with the usual Cartan involution on its Lie algebra $\theta(X) =-X^t$ and Cartan decomposition $\mf{g}=\mf{k}\oplus\mf{p}$ into skew-symmetric and symmetric traceless real $3\times 3$ matrices. Let
    \begin{align*}
        K=\begin{pmatrix}
            &1& \\
            -1&& \\
            &&
        \end{pmatrix}
        \:,\quad P = \begin{pmatrix}
            1 & & \\
            & 1 & \\
            & & -2
        \end{pmatrix}.
    \end{align*}
    \noindent Then $\mf{h}^+\coloneqq \sp_\R\lbrace K \rbrace$ is a Cartan subalgebra of $\mf{k}$, and its centralizer is $\mf{h} = \sp_\R\lbrace K,P\rbrace$. Complexifying, we have $\mf{h}_\C = \sp_\C\lbrace K,P\rbrace$. One can check that the adjoint action of $K$ on $\mf{k}_\C$ has eigenvalues $0,i,-i$ with respective eigenspaces
\begin{align*}
    \mf{h}^+_\C, \quad \sp_\C  \begin{pmatrix}
        & & 1\\
        & & i\\
        -1&-i&
    \end{pmatrix}, \quad 
    \sp_\C  \begin{pmatrix}
        & & 1\\
        & & -i\\
        -1&i&
    \end{pmatrix} .
\end{align*}
In particular, we may choose a positive root system $\Phi^+_k\coloneqq \lbrace \phi\rbrace$ for $\mf{h}^+_\C$ with $\phi:\mf{h}^+_\C\to \C$ given by $\phi(K) = i$. 

Consider now the dual space $\mf{h}^*_\C$ and the dual basis $\lbrace K^*,P^*\rbrace$ defined by $K^*(K) = 1$, $K^*(P) = 0$ and analogously for $P^*$. One can then check that the six roots of $\mf{h}_\C$, in terms of this basis, are $\pm( 2i,0), \pm( i,3)$ and $\pm( i,-3)$. As the Cartan involution acts by precomposition, we see that the dual basis inherits the action of the Cartan involution from the original basis, i.e. $\theta(K^*) = K^*$ and $\theta(P^*)=-P^*$.

In particular, the action of the Cartan involution on a vector written in the dual basis is $(a,b)\mapsto (a,-b)$. We then have exactly one choice of a compatible positive root system for $\mf{h}$: As it must restrict to $\Phi_k^+$, we must pick $(i,3)$ or $(i,-3)$, and as we have to be closed under the Cartan involution we have to pick both. Thus, by the axioms of positive root systems, we see that $\Phi^+ = \lbrace (i,3),(i,-3),(2i,0)\rbrace$.     
\end{exmp}

\noindent We return to the general setting. As the adjoint action of $\mf{k}$ preserves the Killing form on $\mf{p}$, which defines $\mf{so}(\mf{p})$, we get a natural map $\mf{k}\to\mf{so}(\mf{p})$, and thus, we get an induced representation of $\mf{k}$ on $S\coloneqq\text{Spin}(\mf{p})$. In the proof of (\cite{BV}, Lemma $4.1$), it is noted that every highest weight of $F\otimes S$ is of the form $\frac12(\mu+\theta\mu)+w\rho-\rho_k$ for some weight $\mu$ of $F$ and some $w\in W^1$. We use this to give a more precise statement.

\begin{lem}\label{weight}
    Any highest weight of $F\otimes S$ is always of the form $\frac12(\nu+\theta\nu)+w\rho-\rho_k$ for $\nu$ the \textit{highest} weight of $F$, and some $w\in W^1$.
\end{lem}

\begin{proof}
    Let $\frac12(\mu_0+\theta\mu_0)+w_0\rho-\rho_k$ be a highest weight of $F\otimes S$, for some $\mu_0$ and $w_0\in W^1$. We show that this $\mu_0$ is $\nu$. Let $M(\mu)\coloneqq \frac12(\mu+\theta\mu)+w_0\rho-\rho_k$. We claim that if $\mu\succ \mu'$, then $M(\mu)\succ M(\mu')$. Indeed, as the Cartan involution fixes the set of our positive roots, by choice of a compatible root system, it fixes the fundamental Weyl chamber, such that $\mu$ is a positive combination of positive roots if and only if the same is true for $\theta(\mu)$. Now we see that
    $$M(\mu)-M(\mu') = \frac12((\mu-\mu')+\theta(\mu-\mu')),$$
    and so, if we assume $(\mu-\mu')$ is a positive combination, so is $\theta(\mu-\mu')$, and thus so is their half-sum, proving the claim. 
    
    As every irreducible representation has a unique highest weight, we have a unique highest weight $\nu$ for $F$. By the above, we immediately get $M(\nu)\succeq M(\mu_0)$, with equality iff $\nu=\mu_0$.
\end{proof}

\noindent We are now ready to prove Proposition \ref{kstrong}. It is shown in the proof of (\cite{BV}, Lemma $4.1$) that 
\begin{align}\label{BVeigenval}
    \tau(\Omega)-\pi(\Omega)\geq \eta(\tau)
\end{align}
for each irreducible unitary representation $\pi$ of $G(\R)$ satisfying $\Hom_K(\Lambda^p\mf{p}\otimes V_\tau^*,\pi)\neq 0$, where $\eta(\tau)$ is given by
\begin{align}\label{delta}
    \eta(\tau) = |\nu+\rho|^2-\left|\frac12(\mu+\theta\mu)+w\rho\right|^2
\end{align}
for the choice of $\mu$ and $w$ given above (\cite{BV}, ($4.1.2$)), and the norm associated to the inner product induced by the Killing form. By Lemma \ref{delta}, this $\mu$ must be $\nu$, so we have
\begin{align*}
    \eta(\tau) = |\nu+\rho|^2-\left|\frac12(\nu+\theta\nu)+w\rho\right|^2
\end{align*}
Consider $(K_1^*,\dots,K_d^*,P_1^*,\dots,P_s^*)$, an orthonormal basis respecting the decomposition in Lemma \ref{dualdecomp}. With respect to this basis, we will write any weight $\mu$ as $\mu=\mu^++ \mu^-$, where $\mu^+=(\mu^+_1,\dots,\mu^+_d,0,\dots,0)$ and $\mu^-=(0,\dots,0,\mu^-_1,\dots,\mu^-_s)$. Recall the fact that $\mf{k}_\C$, $\mf{p}_\C$ are the eigenspaces of $\theta$ with eigenvalues $1$ and $-1$, respectively. As the dual basis inherits the action from the Cartan involution, we get that $\theta(\mu^++ \mu^-)=\mu^+- \mu^-$. In particular, $\frac12(\mu+\theta\mu) = \mu^+$. Thus, for $\nu=(\nu_1^+,\dots,\nu_d^+,\nu_1^-,\dots,\nu_s^-)$ the highest weight of $F$, expressed in the basis above, we may write
\begin{align}\label{eta}
    \eta(\tau) = |\nu+\rho|^2-\left\vert\frac12(\nu+\theta\nu)+w\rho\right\vert^2 = \sum_{i=1}^s |\nu_i^-|^2+\text{linear terms in }\nu.
\end{align}

\noindent Note that the second equality above follows from the orthogonality of the basis. By (\ref{BVeigenval}), we need only argue that there exists infinitely many $\tau$ such that $\eta(\tau)\geq \lambda$. By the expression above, it is sufficient to be able to find infinitely many representations with highest weight having large $\mf{h}^-$-part. This is possible by the theorem of the highest weight, stating that every dominant integral element is the unique highest weight of an irreducible representation. The dominant integral elements constitute a lattice in the fundamental Weyl chamber associated to $\Phi^+$, which is some cone in the weight space. By the assumption $\delta(G(\R))\geq 1$, we know that $s\geq 1$. 

To be precise, we do the following: Fix any half-line in the fundamental Weyl chamber starting at $0$ and passing through a lattice point, and not lying in the hyperplane given by $x_{d+1}=x_{d+2} = \dots = x_{d+s}=0$. This line is guaranteed to pass through infinitely many lattice points. For any $\lambda>0$ and any linear polynomial $p$ in $d+s$ variables there exists some $r>0$ such that at distance at least $r$ from $0$, every point on this line will have $\sum^s_{i=1}|x_{d+i}|^2-|p(x_1,\dots,x_{d+s})|$ larger than $\lambda$. Considering the expression (\ref{eta}), we see that any lattice point lying on the line with distance at least $r$ to $0$ must have $\eta(\tau)\geq \lambda$. Thus, the line contains infinitely many lattice points associated to $\lambda$-strongly acyclic representations. This finishes the proof of Proposition \ref{kstrong}.

\medskip

\begin{rmk}
    Instead of the final paragraph using half-lines, let us here provide a more visual and geometric argument. Let $G=\SL(3,\R)$ for simplicity. Here, the fundamental Weyl chamber is a cone in two-dimensional space. Writing points $(x,y)$ in terms of the basis given in Example \ref{SL(3,R)example}, we see that by (\ref{eta}), the inequality $\eta(\tau)\geq \lambda$ turns into an inequality of the form
    \begin{align*}
        \lambda \le y^2+ay-bx-c,
    \end{align*}
    for $a,b,c\in\R$. Geometrically, this means that we are considering points outside some parabola. Below we have visualized the positive Weyl chamber as enclosed by the blue lines, the parabola in red and the area outside it in the cone in yellow.
    \begin{center}
        \includegraphics[scale=0.4]{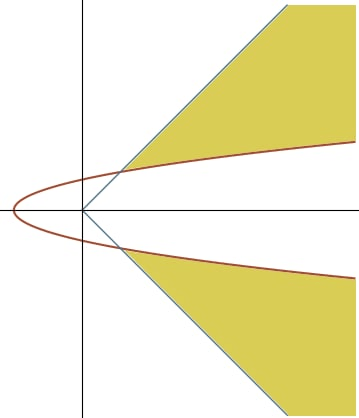}
        
        {\footnotesize\emph{Figure 1. The $\lambda$-strongly acyclic region in the fundamental Weyl chamber of $\SL(3,\R)$.}}
    \end{center}
    From the picture it is clear that the yellow region contains infinitely many lattice points for any lattice in the cone. In fact, it contains \emph{most} of them.
\end{rmk}

\noindent As $\delta(\SL(n,\R))\geq 1$ for $n\geq 3$, we have an immediate corollary to the proposition.

\begin{cor}
    For any $\lambda>0$ and $n\geq 3$, There exists infinitely many $\lambda$-strongly acyclic representations of $\SL(n,\R)$.
\end{cor}

\bigskip

\section{A bound on the heat kernel}\label{plancherel}
\noindent We continue with the same notation as the previous section, such that $G$ is a connected semisimple algebraic group. From now on, we assume that $\tau$ is an irreducible $\lambda$-strongly acyclic representation of $G(\R)$, with $\lambda>0$ to be chosen later. The goal of this section is to present an upper bound on the trace of the heat kernel for large $t$. This will be used to bound certain orbital integrals in Section \ref{asymptotics}.

\subsection*{Induced operators}

Let $(\pi,\mathcal{H}_\pi)$ be a unitary representation of $G(\R)$. For any $f\in \mathcal{C}(G(\R))$, the operator
\begin{align*}
    \pi(f) \coloneqq \int_{G(\R)} f(g)\pi(g)dg
\end{align*}
is a trace class operator on $\mathcal{H}_\pi$. If furthermore $f$ is left- and right $K$-finite, the operator is of finite rank.  We consider 
\begin{align*}
    \pi(H_t^{\tau,p}) \coloneqq \int_{G(\R)} \pi(g)\otimes H_t^{\tau,p}(g)dg
\end{align*} 
as a bounded operator acting on $\mathcal{H}_\pi\otimes \Lambda^p\mf{p}^*\otimes V_\tau$. Let $P$ be the orthogonal projection to the $K$-invariant subspace $(\mathcal{H}_\pi\otimes \Lambda^p\mf{p}^*\otimes V_\tau)^K$. We note that as $\pi$ and $\nu_{\tau,p}$ are unitary, so is their tensor product, and thus $P$ has the form
\begin{align*}
    P = \int_K \pi(k)\otimes v_{\tau,p}(k) dk.
\end{align*}
Using the covariance property, one easily checks that
\begin{align*}
    P\circ \pi(H_t^{\tau,p}) = \pi(H_t^{\tau,p})\circ P = \pi(H_t^{\tau,p}).
\end{align*}
Following the argument of (\cite{BM}, Corollary $2.2$), we can then use (\ref{casimirglobal}) to get
\begin{align}\label{heatcasimir}
    \pi(H_t^{\tau,p}) = e^{-t(\tau(\Omega)-\pi_{\sigma,i\nu}(\Omega))} P.
\end{align}
We will need the following lemma.
\begin{lem}\label{tracecommutes}
    Let $(\pi,\mathcal{H}_\pi)$ be an admissible unitary representation of $G(\R)$, and let $A:\mathcal{H}_\pi\to \mathcal{H}_\pi$ be a bounded operator. Then $(A\otimes 1)\circ \pi(H_t^{\tau,p})$ is of trace class, and
    \begin{align*}
        \Tr \left((A\otimes 1)\circ \pi(H_t^{\tau,p})\right) = \Tr\left(A\circ \pi(h_t^{\tau,p})\right),
    \end{align*}
    Here, the $\Tr$ on the left-hand side is the trace over ${\mathcal{H}_\pi\otimes\Lambda^p\mf{p}^*\otimes V_\tau}$, while on the right-hand side it is the trace over $\mathcal{H}_\pi$. 
\end{lem}

\begin{proof}
    By (\ref{heatcasimir}), we may restrict computing the trace of $(A\otimes 1)\circ \pi(H_t^{\tau,p})$ to computing it over $(\mathcal{H}_\pi\otimes \Lambda^p\mf{p}^*\otimes V_\tau)^K$, and as $\pi$ was assumed admissible, this is finite-dimensional, thus the trace is well defined and finite. The equality now follows by arguing as in the proof of (\cite{BM}, Lemma $5.1$).
\end{proof}

\subsection*{Application of the Plancherel formula}

Let $P=MAN$ be a real standard parabolic subgroup of $G(\R)$. Let $\mf{a}$ be the Lie algebra of $A$. We denote by $\langle\cdot,\cdot\rangle$ the inner product on the real vector space $\mf{a}^*$ induced by the Killing form, and $||\cdot||$ its associated norm. Fix a restricted positive root system of $\mf{a}$ and let $\rho_{\mf{a}}$ denote their half sum. Let $(\sigma,W_\sigma)$ be a discrete series representation of $M$, i.e. an irreducible unitary subrepresentation of the left regular representation of $M$ on $L^2(M)$, and let $i\nu\in i\mf{a}^*$. We denote by $(\pi_{\sigma,i\nu},\mathcal{H}^{\sigma,i\nu})$ the induced \emph{principal series representation} of $G$, defined by
\begin{align*}
    \mathcal{H}^{\sigma,i\nu}=\lbrace f:G\to W_\sigma \mid f(gman) = a^{-(\nu+\rho_{\mf{a}})}\sigma(m)^{-1}f(g) \\
    \:\forall g\in G(\R), man\in MAN,\:\: f|_{K}\in L^2(K,W_\sigma)&\rbrace, \\
    (\pi_{\sigma,\nu}(g)f)(x) = f(g^{-1} x).\hspace{16.8em}
\end{align*}
This is an irreducible unitary representation of $G(\R)$, in particular admissible, and thus by Lemma \ref{tracecommutes}, we have
\begin{align}\label{traceheat}
    \Tr(\pi_{\sigma,\nu}(H_t^{\tau,p})\pi_{\sigma,\nu}(g)) = \Tr(\pi_{\sigma,\nu}(h_t^{\tau,p})\pi_{\sigma,\nu}(g)).
\end{align}
As in the lemma, the traces are over the appropriate spaces. Applying (\ref{heatcasimir}) to the above, we get that 
\begin{align}\label{planchereltrace}
    \Tr(\pi_{\sigma,\nu}(h_t^{\tau,p})\pi_{\sigma,\nu}(g)) = e^{-t(\tau(\Omega)-\pi_{\sigma,i\nu}(\Omega))} \Tr(\pi_{\sigma,i\nu}^K(g)),
\end{align}
where we by $\pi_{\sigma,i\nu}^K(g)$ denote the operator $P(\pi_{\sigma,i\nu}(g)\otimes \text{Id})P$ on $\mathcal{H}^{\sigma,i\nu}\otimes \Lambda^p\mf{p}^*\otimes V_\tau$, with $P$ the projection onto the $K$-fixed vectors. We consider this as an operator on $(\mathcal{H}^{\sigma,i\nu}\otimes \Lambda^p\mf{p}^*\otimes V_\tau)^K$, as this subspace contains its image and it is $0$ elsewhere. This subspace is a finite-dimensional space, and by Frobenius reciprocity we have that
\begin{align*}
    \dim(\mathcal{H}^{\sigma,i\nu}\otimes \Lambda^p\mf{p}^*\otimes V_\tau)^K = \dim(W_\sigma\otimes \Lambda^p\mf{p}^*\otimes V_\tau)^{K_M},
\end{align*}
with $K_M\coloneqq K\cap M$. By virtue of being unitary, we have the inequality
\begin{align}\label{dimupper}
    |\Tr(\pi_{\sigma,i\nu}^K(g))|\leq \dim(W_\sigma\otimes \Lambda^p\mf{p}^*\otimes V_\tau)^{K_M}.
\end{align}
With this bound established, we turn to the scalar appearing in (\ref{planchereltrace}). Define a constant $c(\sigma)$ by
\begin{align*}
    c(\sigma) \coloneqq \sigma(\Omega_M)-||\rho_{\mf{a}}||.
\end{align*}
Then we have that (see \cite{Knapp}, Proposition $8.22$)
\begin{align*}
    \pi_{\sigma,i\nu}(\Omega) = c(\sigma)-||\nu||^2.
\end{align*}
Assume that $\dim(\mathcal{H}^{\sigma,i\nu}\otimes \Lambda^p\mf{p}^*\otimes V_\tau)^K\neq 0$. By the assumption that $\tau$ is $\lambda$-strongly acyclic, we immediately get that
\begin{align*}
    \tau(\Omega)-c(\sigma)\geq \lambda-||\nu||^2,
\end{align*}
and as this is true for all $\nu\in\mf{a}^*$, we get
\begin{align}\label{gap}
    \tau(\Omega)-c(\sigma)\geq \lambda.
\end{align}

\noindent Now we are ready to prove a large $t$ asymptotic for $h_t^{\tau,p}$. Using the Harish-Chandra Plancherel formula (see \cite{Olbrich}, Theorem $2.2$) and (\ref{traceheat}) we get that
\begin{align*}
    h_t^{\tau,p}(g) &= \sum_P\sum_{\sigma\in \hat{M}_d}\int_{\mf{a}^*}\Tr(\pi_{\sigma,i\nu}(h_t^{\tau,p})\pi(g^{-1}))p_\sigma(i\nu) d\nu \\
    &=e^{-t(\tau(\Omega)-c(\sigma))}\sum_P\sum_{\sigma\in \hat{M}_d}\int_{\mf{a}^*} e^{-t||\nu||^2}\Tr(\pi_{\sigma,i\nu}^K(g))p_\sigma(i\nu) d\nu,
\end{align*}
where $p_\sigma:i\mf{a}^*\to [0,\infty)$ is the Plancherel density, an analytic function of polynomial growth. The first sum runs over $P=MAN$ real standard parabolic subgroups, and the second over discrete series representations of $M$ the Levi subgroup of $P$. Since the trace inside the integral on the right-hand side is bounded absolutely by $\dim(W_\sigma\otimes \Lambda^p\mf{p}^*\otimes V_\tau)^{K_M}$, and this is non-zero only for finitely many pairs $(P,\sigma)$ (see \cite{Olbrich}, Corollary $2.3$), we may take the double sum to be finite. Furthermore, which pairs contribute is governed solely by $\tau$ and $p$. 

The latter integral is convergent and vanishing for $t\to\infty$, and so for $t\geq 1$ it may be bounded independently of $t$. Putting these observations together and applying (\ref{gap}), we get the following result.
\begin{prop}\label{traceheatdecay}
    Assume $t\geq 1$. Then there exists a constant $C>0$ only depending on $\tau$ and $p$ such that for any $g\in G(\R)$, we have
    \begin{align*}
        | h_t^{\tau,p}(g)|\leq C\,e^{-\lambda t}.
    \end{align*}
\end{prop}

\noindent Later, we will also need a vanishing behaviour of $h_t^{\tau,p}(g)$ in terms of $g$. This has been explored in \cite{LéM} for the heat kernel associated to Laplace operators on standard $p$-forms. By following the proof of their Theorem $3.1$, adapted to our setting, we get a refined bound on the heat kernel, which turns into the following bound on its trace. Let $r(g)=d(gK,K)$ be the geodesic distance from $gK$ to $K$ on $G(\R)/K$. Then there exists constants $A,c,C>0$ such that
\begin{align}\label{vanishin|g|}
    |h_t^{\tau,p}(g)|\leq C\, e^{-\lambda t}\, e^{-c\frac{r(g)^2}{t}}
\end{align}
for all $t>1$ and $g\in G(\R)$ satisfying $r(g) > A$.

\bigskip

\section{Analysis of the trace formula}\label{trace formula}

\subsection*{Review of the geometric side of the trace formula}\label{tracereview}

Let $G$ be a reductive algebraic group over $\Q$ with $G(\R)$ noncompact. Fix a minimal parabolic subgroup $P_0$ of $G$ with a Levi decomposition $P_0=M_0 N_0$. For $\GL(n)$, we pick as minimal parabolic subgroup the subgroup of upper triangular matrices, with $M_0$ the diagonal matrices in $G$. We set $\mathcal{F}$ to be the set of parabolic subgroups of $G$ defined over $\Q$ containing $M_0$. We let $\mathcal{L}$ denote the set of subgroups in $G$ containing $M_0$ that are also Levi components of some group in $\mathcal{F}$. Furthermore, any $L\in\mathcal{L}$ is a reductive group, and for $M\in\mathcal{L}$ a Levi subgroup we shall denote by $\mathcal{L}^L(M)$ the set of Levi subgroups in $L$ containing $M$. Finally we will write $\mathcal{P}^L(M)$ for the set of parabolic subgroups of $L$ for which $M$ is a Levi component. If $L=G$, we drop the superscript and write $\mathcal{L}(M)$ and $\mathcal{P}(M)$.

Let $\A_f$ be the finite adeles over $\Q$. Given $K_f\subset G(\A_f)$ an open compact subgroup, we can define the adelic Schwartz space $\mathcal{C}(G(\A)^1,K_f)$ as the space of smooth right $K_f$-invariant functions on $G(\A)^1$ all of whose derivatives lies in $L^1(G(\A)^1)$. We denote by $\mathcal{C}(G(\A)^1)$ the union of $\mathcal{C}(G(\A)^1,K_f)$ over all such $K_f$.

For $f\in C_c^\infty(G(\A)^1)$, let $J_{\geo}(f)$ be the geometric side of the Arthur trace formula (see \cite{Arthur2}). This is a distribution with test function $f$. We give a very brief sketch of its construction (see \cite{Arthur0} for an excellent introduction). In essence, we wish to integrate the function
\begin{align}\label{naivekernel}
    K(x,y) = \sum_{\gamma\in G(\Q)}f(x^{-1}\gamma y), \qquad x\in G(\A)^1
\end{align}
over $G(\Q)\backslash G(\A)^1$. In our non-compact case however, this function is often not integrable, and lacks some of our desired properties. Correction terms must be added, and it turns out to be a good idea to add one for each parabolic subgroup. Let $P\in \mathcal{F}$ with canonical Levi decomposition $P=M_PN_P$, i.e. such that $M_P\in \mathcal{L}(M_0)$. We then define
\begin{align*}
    K_P(x,y)\coloneqq\int_{N_P(\A)} \sum_{\gamma\in M_P(\Q)}f(x^{-1}\gamma n y) dn, \qquad x,y\in G(\A)^1.
\end{align*}
From this, one constructs a kernel function $k^T(x,f)$ as a sum over standard parabolic subgroups and with a truncation parameter $T$ which serves as a replacement for the function (\ref{naivekernel}), see (\cite{Arthur2}) for a precise definition. One then defines $J^T(f)$ as the $G(\Q)\backslash G(\A)^1$-integral of $k^T(f,x)$. Finally, one picks a particular truncation parameter $T=T_0$ and defines $J(f)\coloneqq J^{T_0}(f)$. In the case of $G=\GL(n)$, the canonical choice is $T_0 = 0$. 

Recall that any element of an algebraic group has a Jordan decomposition, i.e. for any $g\in G(k)$ with $k$ a perfect field, we have a decomposition $g=g_sg_u=g_ug_s$ with $g_s$ semisimple and $g_u$ unipotent. The terms \textit{semisimple} and \textit{unipotent} simply mean that their image has this property under some (equivalently, any) closed embedding $G\to \GL(n)$. Now, define an equivalence relation on $G(\Q)$: Say that two elements are equivalent if their semisimple parts are $G(\Q)$-conjugate to each other. Let $\mathcal{O}$ be the set of equivalence classes in $G(\Q)$. Then for $\mf{o}\subset \mathcal{O}$, one could also consider the function
\begin{align}\label{classkernel}
    K_{P,\mf{o}}(x,y)\coloneqq\int_{N_P(\A)} \sum_{\gamma\in M_P(\Q)\cap \mf{o}}f(x^{-1}\gamma n y) dn, \qquad x,y\in G(\A)^1.
\end{align}
From this, one can again construct a modified kernel function $k_{\mf{o}}(x,f)$ as above, and define another distribution with test function $f$ analogously that we will denote $J_\mf{o}(f)$ (see \textit{loc. sit.}). We see that we have an equality
\begin{align*}
    K_P(x,y) = \sum_{\mf{o}\in\mathcal{O}}K_{P,\mf{o}}(x,y).
\end{align*}
This equality directly extends to the identity on kernel functions $k^T(f,x)=\sum_{\mf{o}\in \mathcal{O}}k^T_{\mf{o}}(f,x)$ once proper definitions are given. The content of the \emph{coarse geometric expansion} is the fact that the analogous identity holds true for the associated distributions. Precisely, it is the equality
\begin{align}\label{coarsegeom}
    J_{\geo}(f) = \sum_{\mf{o}\in\mathcal{O}}J_{\mf{o}}(f).
\end{align}

\noindent The coarse geometric expansion has been extended to the domain $\mathcal{C}(G(\A)^1,K_f)$ in (\cite{FL}), which will be important for our applications.

We note that the distributions $J_{\mf{o}}(f)$ have a formulation in terms of integrals of $f(x^{-1}\gamma x)$, for $\gamma\in\mf{o}$ and $x\in G(\A)$ (e.g. \cite{Arthur5}, Theorem $8.1$). In particular, $J_{\mf{o}}(f)$ is non-zero only if the union of $G(\A)$-conjugacy classes of $\mf{o}$ intersects the support of $f$. This will be used in Section \ref{compactsub} to show that in some cases, only one particular class contribute.

\subsection*{Analytic torsion}\label{analytictorsionsection} We assume that $G$ is $\GL(n)$ or $\SL(n)$, and let $K_f$ and $X(K_f)$ be defined as in Section \ref{arithmfd}, with $K_f$ neat. We denote by $h_t^{\tau,p}$ the trace of the heat kernel as defined in (\ref{trheat}). Set $1_{K_f}$ as the indicator function of $K_f$ on $G(\A_f)$, and define
\begin{align*}
    \chi_{K_f}\coloneqq\frac{1_{K_f}}{\vol(K_f)}.
\end{align*} Then $h_t^{\tau,p}\otimes \chi_{K_f}\in \mathcal{C}(G(\A)^1,K_f)$. Following (\cite{MzM1}), we define the regularized trace of the heat operator as
\begin{align*}
    \Trreg\left(e^{-t\Delta_p(\tau)}\right)\coloneqq J_{\geo}\left(h_t^{\tau,p}\otimes \chi_{K_f}\right).
\end{align*}

\noindent If $X(K_f)$ is compact, the heat operator is of trace class, and the regularized trace defined above is then equal to the usual trace. We define the associated spectral zeta functions $\zeta_p(s,\tau)$ as in the compact case (\ref{zeta}), replacing the usual trace with the regularized version. Absolute convergence and existence of meromorphic continuation was shown in (\cite{MzM1}). However, unlike the compact case, the zeta function might have a pole at $s=0$, and we have to be slightly more careful. For a meromorphic function $f(s)$ on $\C$ and $z\in \C$, we define $\FP_{s=a}f$ as the zeroth coefficient of the Laurent expansion of $f(s)$ at $s=a$. Analytic torsion of $X(K_f)$ is then defined by
\begin{align}\label{torsion}
    \log T_{X(K_f)}(\tau) = \frac12\sum_{p=0}^d(-1)^p\;p\; \FP_{s=0}\left(\frac{\zeta_p(s,\tau)}{s}\right).
\end{align}

\subsection*{Congruence quotients of $\GL(n)$ and $\SL(n)$}\label{congquot}

We stay in the same setup as the previous subsection. Recall that if $K_f$ is neat, $X(K_f)$ is a locally symmetric manifold of finite volume. This assumption is true in the primary setting of this paper that we now explain. Let $N\in\N$, $N\geq 3$ and define 
\begin{align}\label{glnprincipalcongruencesubgroup}
    K(N)\coloneqq\prod_p K_p\left(p^{\nu_p(N)}\right) \subset \GL(n,\A_f),
\end{align}
where $K_p(p^e)$ is the kernel of the canonical map $\GL(n,\Z_p)\to \GL(n,\Z/p^e\Z)$. Then $K(N)$ is an open compact subgroup of $\GL(n,\A_f)$. We further define $K'(N)$ to be the completely analogous subgroup for $\SL(n)$. Pick the maximal compact subgroups of $\GL(n,\R)^1$ and $\SL(n,\R)$ to be $\text{O}(n)$, respectively $\SO(n)$. We now set $Y(N)$ to be the adelic symmetric space for $\GL(n)$ associated to $K(N)$ as in (\ref{adelicsym}), and similarly let $X(N)$ be the space for $\SL(n)$ associated to $K'(N)$.

It follows from Section \ref{arithmfd} that $X(N)\cong \Gamma(N)\backslash \SL(n,\R)/\SO(n)$ with $\Gamma(N)$ the standard principal congruence subgroup of $\SL(n,\Z)$ of level $N$, and also that $Y(N)$ is a disjoint union of $\phi(N)$ many copies of $X(N)$. Thus, it is reasonable to also call $N$ the level of the subgroup $K(N)$. In particular, we get
\begin{align}\label{volglsl}
    \vol(Y(N)) = \phi(N)\vol(X(N)).
\end{align}
We can consider the associated analytic torsion as defined in (\ref{torsion}) in both of these settings. As one would hope for, by (\cite{MzM2}, ($11.5$)) we have
\begin{align}\label{torsionglsl}
    \log T_{Y(N)}(\tau) = \phi(N)\log T_{X(N)}(\tau).
\end{align}
Combining (\ref{volglsl}), (\ref{torsionglsl}) and (\ref{l2torsion}), we see that Theorem \ref{mytheorem} is equivalent to the following.
\begin{thm}\label{glmytheorem}
    Assume $\tau$ is a $\lambda$-strongly acyclic representation of $\GL(n,\R)^1$, for a certain $\lambda$ depending only on $n$. Then there exists some $a>0$ such that
    $$\log T_{Y(N)}(\tau) = \log T_{Y(N)}^{(2)}(\tau)+O(\vol(Y(N))N^{-(n-1)}\log(N)^a)$$
    as $N$ tends to infinity.
\end{thm}
\noindent The remainder of the paper is essentially a proof of this theorem.

\subsection*{Compactification of the test function}\label{compactsub}

\noindent We continue with the same setup, and assume that $G=\GL(n)$ and $K_f=K(N)$, $N\geq 3$. Let $r(g)=d(K,gK)$ be the geodesic distance of $g\in G(\R)^1$ from the identity on $\Tilde{X}=G(\R)^1/K$. We will often write $d(I,g)$ for the same expression. Let $\phi_R:G(\R)^1\to [0,1]$ be a smooth function identically $1$ on $B(R)=\lbrace g\in G(\R)^1\mid r(g)< R\rbrace$ and identically $0$ outside $B(R+\epsilon)$ for some small $\epsilon>0$. We now define
\begin{align*}
    h_{t,R}^{\tau,p}(g)\coloneqq \phi_R(g)h_t^{\tau,p}(g).
\end{align*}

\noindent In a moment, we will need this small lemma on this distance function.

\begin{lem}\label{distancebound}
    Let $||\cdot||$ denote the Frobenius norm on $M_{n\times n}(\R)$. Then
    \begin{align*}
        r(g) \geq \log ||g||, \quad \forall g\in G(\R)^1.
    \end{align*}
\end{lem}

\begin{proof}
    By (\cite{bridsonhaefliger}, II. Corollary 10.42(2)), we have that $d(I,e^X) = ||X||$ for $X$ a symmetric matrix. Further, by the basic inequality $e^{||X||}\geq ||e^X||$, we see that
\begin{align*}
    d(I,e^X)\geq \log||e^X||.
\end{align*}
Since both the metric and the Frobenius norm are invariant under multiplication by orthogonal matrices, this inequality holds when replacing $e^X$ with any $g\in G(\R)^1$, using the Cartan decomposition $G(\R)^1=KAK$ with $A$ the set of diagonal matrices in $G(\R)^1$.
\end{proof}

\noindent Let $J_{\unip}$ denote the distribution defined in Section \ref{tracereview} associated to the equivalence class of elements with semisimple part being the identity, i.e. the unipotent elements. Replacing $h_t^{\tau,p}$ with $h_{t,R}^{\tau,p}$ allows us to reduce the geometric side of the Arthur trace formula to only the unipotent contribution, if we keep $R$ small relative to the level $N$. This is the content of the following theorem.

\begin{prop}\label{compactunip}
    For $N$ large enough and $R\leq C_n\log N$, the constant $C_n>0$ only depending on $n$, we have that
    \begin{align*}
        J_{\geo}(h_{t,R}^{\tau,p}\otimes\chi_{{X(N)}}) = J_{\unip}(h_{t,R}^{\tau,p}\otimes\chi_{{X(N)}}).
    \end{align*}
\end{prop}

\begin{proof}
\noindent  In the coarse geometric expansion of the Arthur trace formula (\ref{coarsegeom}), we are summing distributions indexed over equivalence classes in $G(\Q)$. For our test function $f$, we wish to pick $R$ such that $J_{\mf{o}}(f)=0$ for any $\mf{o}$ not the unipotent class. As explained at the end of Section \ref{tracereview}, for this it is sufficient to show that the $G(\A)$-conjugacy classes of $\mf{o}$ do not intersect the support of the test function.

Our test function $f=h^{\tau,d}_{t,R}\otimes \chi_{K(N)}$ has its support inside $B_R K(N)\subset G(\A)$. We will pick $R$ such that for any $\gamma$ with semisimiple part $\gamma_{ss}\neq I$ in $G(\Q)$, every conjugate lies outside the support. Take such a $\gamma$, and let $g = x^{-1}\gamma x$ for some $x\in G(\A)$. Write $g=g_\infty\prod_pg_p$ and assume $\prod_pg_p\in K(N)$. We will then pick $R$ such that $g_\infty\notin B_R$.

Let $q(x)\in\Q[x]$ be the characteristic polynomial of $\gamma-I$. Note that by conjugation invariance, this is also the characteristic polynomial of $g_\nu-I$ for all places $\nu$. As $\gamma$ is not unipotent by assumption, this polynomial has a non-leading, non-zero coefficient, say for the degree $k$ term, $0\leq k\leq n-1$. Recall that this polynomial is $q(x)=\det(xI-(g_\nu-I))$. Then for every prime $p$ this coefficient, call it $a_k$, satisfies
\begin{align*}
    \nu_p(a_k) \geq (n-k)\cdot\nu_p(N),
\end{align*}
since it is a sum of products of $n-k$ elements of $p^{\nu_p(N)}\Z_p$, as $g_p$ lies in $K_p(p^{\nu_p(N)})$. In particular, it is integral, and as it is non-zero we get that $|a_k|\geq N^{n-k}$. This implies that at least one of the entries of $g_\infty-I$ has norm greater than $c_n N$, for some constant $c_n>0$ only depending on $n$, and this in turn implies the same lower bound on the Frobenius norm of $g_\infty-I$. Applying Lemma \ref{distancebound}, we get that there exists some $C_n>0$ such that
$$r(g_\infty)=d(I,g_\infty)\geq \log||g_\infty||\geq C_n\log N$$
for $N$ large enough, only depending on $n$. The last inequality is just using the reverse triangle inequality on our lower bound on $||g_\infty-I||$. Thus, it is clear that picking $R\leq C_n\log N$, we have that $g_\infty\notin B_R$ as desired.
\end{proof}

\noindent We must ensure that in replacing $h_t^{\tau,p}$ with its compactification, we can control the change in the trace formula. This was shown in \cite{MzM2}:

\begin{prop}[\cite{MzM2}, Proposition $7.2$]\label{compactdiff}
    There exists constants $C_1,C_2,C_3>0$ such that
    \begin{align*}
        \big\vert J_{\spec}(h_t^{\tau,p}\otimes\chi_{{X_N}})-J_{\spec}(h_{t,R}^{\tau,p}\otimes\chi_{{X_N}}))\big\vert\leq C_3 e^{-C_1 R^2/t+C_2t}\vol(Y(N))
    \end{align*}
    for all $t>0$, $R\geq 1$ and $N\in\N$.
\end{prop}

\noindent It is in controlling this error term that we need to vary $R$ with $N$. The specifics will be discussed in Section \ref{dance}. Importantly, we may pick $C_1,C_2$ independent of the representation $\tau$.

\subsection*{The fine geometric expansion}\label{finegeomexppaper}

By the coarse geometric expansion (\ref{coarsegeom}), Proposition \ref{compactunip} and Proposition \ref{compactdiff}, when analyzing $\Trreg(e^{-t\Delta_p(\tau)})$ we may restrict our attention to $J_{\unip}(h^{\tau,p}_{t,R}\otimes\chi_{K(N)})$. This distribution can be expressed as a finite sum of weighted orbital integrals with certain coefficients, a result known as the fine geometric expansion (\cite{Arthur4}, Corollary $8.3$). First we need a bit of notation.

Let $S$ be a finite set of primes containing $\infty$, and set $\Q_S=\prod_{\nu\in S}\Q_\nu$ and $\Q^S=\prod_{\nu\notin S}'\Q_\nu$, with $\prod'$ the usual restricted product. We define $G(\Q_S)^1 \coloneqq G(\Q_S)\cap G(\A)^1$, and write $C_c^\infty(G(\Q_S)^1)$ for the space of functions $C_c^\infty(G(\Q_S))$ restricted to $G(\Q_S)^1$.

Given $M\in \mathcal{L}$, denote by $\mathcal{U}_M(\Q)$ the finite set of conjugacy classes of unipotent elements of $M(\Q)$. Assume $f=f_S\otimes 1_{K^S}$, with $f_S\in C_c^\infty(G(\Q_S)^1)$ and $1_{K^S}$ the characteristic function of the standard compact subgroup in $G(\Q^S)$.
In the case of $\GL(n)$, as all orbits are stable, the fine geometric expansion can then be expressed as 
\begin{align}\label{finegeom}
    J_{\text{unip}}(f) = \sum_{M\in\mathcal{L}}\:\sum_{\mf{u}\in \mathcal{U}_M(\Q)}a^M(S,\mf{u})J_M(f_S,\mf{u}).
\end{align}
Here $J_M(f,\mf{u})$ is the weighted orbital integral associated to $(M,\mf{u})$ of $f$, and $a^M(S,\mf{u})$ are the \emph{global coefficients}. One finds the general definition of weighted orbital integrals in \cite{Arthur1} - we will express them in our specific situation in a moment. We will apply this to our test function $h^{\tau,p}_{t,R}\otimes \chi_{K(N)}$. We will by abuse of notation write $\chi_{K(N)}$ both for the normalised characteristic function of $K(N)$ in $G(\A_f)$ and in $G(\Q_{S(N)})$, where $S(N)=\lbrace p\text{ prime} : p\mid N\rbrace$. In this case, $a^M(S(N),\mf{u})$ depend on $N$ only by its prime divisors, and does not grow too quickly, as seen from the following lemma.

\begin{lem}[\cite{Matz}]\label{coeffbound}
    There exists $b,c>0$ such that for all $N$, $M$ and $\mf{u}$ we have
    \begin{align*}
        |a^M(S(N),\mf{u})|\leq c(1+\log N)^b.
    \end{align*}
\end{lem}

\noindent To describe the weighted orbital integrals, we split them into archimedean and non-archimedean parts (see \cite{Arthur3}). Assume that $f=f_\infty\otimes f_{\text{fin}}$. For $L\in\mathcal{L}(M)$ and $Q=LV\in\mathcal{P}(L)$, define
\begin{align}\label{subq}
    f_{\infty,Q}(m) \coloneqq \delta_Q(m)^{\frac12}\int_{K_\infty}\int_{V(\R)}f_\infty(k^{-1}mvk)dv dk, \quad m\in M(\R).
\end{align}
Define $f_{\text{fin},Q}$ analogously. Then we have the expression
\begin{align}\label{localparts}
    J_M(f,\mf{u}) = \sum_{L_1,L_2\in\mathcal{L}(M)} d_M^G(L_1,L_2)J^{L_1}_M(f_{\infty,Q_1},\mf{u}_\infty)J_M^{L_2}(f_{\text{fin},Q_2},\mf{u}_{\text{fin}}).
\end{align}
Here $Q_i$ is some parabolic in $\mathcal{P}(L_i)$, and $\mf{u}_{\text{fin}}=(\mf{u}_p)_p$, where $\mf{u}_p$ the $M(\Q_p)$-conjugacy class of $\mf{u}$. Denote $\mf{u}_\infty$ analogously. The archimedean part is an integral of the form (see \cite{Arthur1})
\begin{align}\label{archilocalorbint}
    J^{L}_M(f_{\infty,Q},\mf{u}_\infty) = \int_{U(\R)}f_{\infty,Q}(u)\omega(u) du
\end{align}
for a certain weight function $\omega$ depending on the class $\mf{u}$, as well as on $M$ and $L$. Here, $U=U_{Q_1}$ is the unipotent radical of a semistandard parabolic subgroup $S=M_{Q_1}U_{Q_1}$ in $M$ such that ${Q_1}$ is a Richardson parabolic for $\mf{u}$ in $M$. By splitting up the finite part further into local parts and computing explicitly, Matz and Müller showed the following lemma.
\begin{lem}[\cite{MzM2},$(9.4)$]\label{localbound}
There exist $c,d>0$ only depending on the group $G$ such that
    \begin{align*}
   \left\vert J_M^L(\chi_{{K(N)},Q},\mf{u}_{\text{fin}})\right\vert \leq c N^{-\dim\Ind^G_M \mf{u}/2}(\log N)^d\vol(K(N))^{-1}.
\end{align*}
\end{lem}

\begin{rmk}
    Unless $(M,\mf{u})=(G,\lbrace 1\rbrace)$, we have that $\frac{\dim\Ind^G_M \mf{u}}{2}\geq n-1$, and this is where the power saving in our main result comes from.
\end{rmk}
\noindent As the non-archimedean part does not see the variable $t$, nor does it see the representation $\tau$ or the radius $R$, we may focus on the integrals
\begin{align}\label{inforbitcrude}
    J^{L}_M(((h^{\tau,p}_{t,R})_Q,\mf{u}_\infty) = \int_{U(\R)}(h^{\tau,p}_{t,R})_Q(u)\omega(u) du.
\end{align}
Recalling the definition of $f_Q$, we may use that $h^{\tau,p}_{t,R}$ is bi-$K_\infty$-invariant, and letting $MV'\coloneqq M_{Q_1}U_{Q_1}V$ such that $V'$ is the unipotent radical of a Richardson parabolic for $\Ind_M^G(\mf{u})$ in $G$, we may rewrite this as
\begin{align}\label{inforbit}
    J^{L}_M(((h^{\tau,p}_{t,R})_Q,\mf{u}_\infty) = \int_{V'(\R)}h^{\tau,p}_{t,R}(v)\omega(v) dv.
\end{align}

\noindent This description will be used in the following section.

\bigskip

\section{Asymptotic time behaviour}\label{asymptotics}
\noindent We continue in the setting of Section \ref{trace formula} with $G=\GL(n)$ and $K_f=K(N)$. It is clear from the definition that to understand analytic torsion, we must analyze the terms
\begin{align}\label{mellin}
    \FP_{s=0}\left(\frac{1}{s\Gamma(s)}\int_0^\infty \Trreg(e^{-t\Delta_p(\tau)})t^{s-1}dt\right),
\end{align}
where we by the integral in fact means its meromorphic continuation to all of $\C$ as a function of $s$. As $\Gamma(s)$ has a simple pole at $s=0$ with residue $1$, we see that $\frac{1}{s\Gamma(s)}$ is holomorphic at $s=0$ with value $1$. Thus, we are reduced to examining the meromorphic continuation of the Mellin transform of $\Trreg(e^{-t\Delta_p(\tau)})$.

From the standard theory of Mellin transforms (see e.g. \cite{Zagier}), we can understand the meromorphic continuation of the Mellin transform of a function $f$ by giving asymptotics of $f(t)$ for $t\to 0$ and $t\to\infty$. More precisely, it is sufficient to establish the following:
\begin{align}\label{largef}
    f(t) &= O(e^{-ct}) \qquad\qquad\qquad\qquad\qquad\qquad \text{as }t\to\infty \\\label{smalltf}
    f(t) &= \sum_{i=0}^B\sum_{j=0}^{r_i}c_{ij}t^{\alpha_i}(\log t)^j +O(t^{\alpha_{B+1}}) \qquad \text{as }t\to 0
\end{align}
for $(\alpha_i)_{i\in\N_0}$ a sequence of real, possibly negative numbers with $\alpha_i<\alpha_{i+1}$ and tending to $+\infty$, and $c>0$. Then we know that the Mellin transform of $f$ converges in some half plane and has a meromorphic continuation to all of $\C$. Furthermore, its residues at its poles can be described in terms of the coefficients $c_{i,j}$. In the following, we present such asymptotics in the different settings we will need for our proof.

\subsection*{Large $t$ asymptotic of the spectral side}

During the proof of the main theorem, we will need to control large $t$ behaviour of the entire trace formula at once to control the error term incurred from truncating the the Mellin transform. This was done in (\cite{MzM2}, Corollary $6.7$), using the fine spectral expansion of the spectral side of the Arthur trace formula and results on logarithmic derivatives of intertwining operators. The result is
\begin{align}\label{larget}
    \left\vert J_{\text{spec}}(h_t^{\tau,p}\otimes \chi_{K(N)})\right\vert \leq Ce^{-c t}\vol(Y(N))
\end{align}
for some $C,c>0$, for all $t\geq 1$, $p=0,\dots,n$ and $N\in\N$.
\begin{rmk}\label{rmkonconstant}
    Going carefully through the proof of the above result in \cite{MzM2}, one sees that we may pick $c=\lambda(1-\epsilon)$ for any $0<\epsilon<1$, where $\lambda$ is the spectral gap guaranteed by assuming $\tau$ is $\lambda$-strongly acyclic. Indeed, in  (\cite{MzM2}, $(6.20)$) they pick $c=\frac{\lambda}{2}$, but the proof works for any multiple of $\lambda$ with a factor less than $1$.  
\end{rmk} 

\subsection*{Small $t$ asymptotics of orbital integrals}

By the definition of the regularized trace of the heat operator, along with Proposition \ref{compactunip} and the fine geometric expansion (\ref{finegeom}), once we have switched to a compactified test function it is sufficient to establish asymptotics for the weighted orbital integrals. By the decomposition into archimedean and non-archimedean part (\ref{localparts}), as the non-archimedean part does not depend on the variable $t$ in our case, we are reduced to analyzing the archimedean parts.

The desired small-$t$ asymptotics (\ref{smalltf}) for our archimedean weighted orbital integrals given in (\ref{inforbit}) were shown in \cite{MzM1}. Combining (\cite{MzM1}, Proposition $12.3$) and (\cite{MzM1}, $(13.14)$) we get
\begin{prop}\label{orbitsmallt}
Let $M\in \mathcal{L}$ and $\mf{u}\in \mathcal{U}_M(\Q)$ with $(M,\mf{u})\neq (G,\lbrace 1\rbrace)$. For every $N\geq 3$, there is an expansion
    \begin{align*}
        J^G_M((h_{t,R}^{\tau,p})_Q,\mf{u}) = t^{-(d-k)/2}\sum_{j=0}^N\sum^{r_j}_{i=0}c_{ij}(\tau,p)t^{j/2}(\log t)^i+O(t^{(N-d+k+1)/2})
    \end{align*}
    as $t\to 0^+$.
\end{prop}
\noindent Here $k$ is the dimension of the Lie algebra of $V(\R)$, and $c_{ij}(\tau,p)$ are certain coefficients depending only on $i,j,\tau,p$. Note that the results (\cite{MzM1}, Proposition $12.3$) is stated for $M\neq G$, but their proof holds in the more general case above without modification. Furthermore, it is sufficient for our purposes to state the result for $L=G$ as above, since every Levi subgroup $L$ of $\GL(n)$ is canonically isomorphic to a finite direct product of $\GL(m)$'s, $m\leq n$, and the orbital integral splits accordingly.

\begin{rmk}\label{independentR}
    The proof of (\cite{MzM1}, Proposition $12.3$) does not utilize the fact that the support of the test function is compactified, i.e. their proof holds when replacing $h_{t,R}^{\tau,p}$ by $h_t^{\tau,p}$ - this can be seen in the first inequality of Section $12.2$. In particular, neither the coefficients nor the implied constants in the error term depend on the radius of compactification $R$. 
\end{rmk}

\subsection*{Large $t$ asymptotics of orbital integrals}

We now show exponential decay as $t$ goes to infinity of the orbital integrals appearing in (\ref{inforbit}). Throughout, we will assume $t > 1$. Let $A,c,C>0$ be given as in (\ref{vanishin|g|}). Recall that $B(k)$ are the elements $g$ of $G(\R)^1$ with $r(g)< k$. We now decompose $V(\R)$ into disjoint subsets of growing radius:
\begin{align*}
    V(\R) = (V(\R)\cap B(A)) \cup \bigcup_{k=1}^\infty V(\R) \cap (B(A+k)\setminus B(A+k-1)).
\end{align*}
To simplify notation, for $k\geq 1$ write 
\begin{align*}
    D(k) \coloneqq  V(\R) \cap (B(A+k)\setminus B(A+k-1))
\end{align*}
and set $D(0)\coloneqq (V(\R)\cap B(A))$. The point is that for $v\in D(k)$ we then have the bounds $A+k-1\leq r(v)<A+k$. The orbital integral can be decomposed
\begin{align}\label{D(k)decomp}
    \int_{V(\R)} h^{\tau,p}_{t,R}(v)\omega(v)dv = \int_{D(0)} h^{\tau,p}_{t,R}(v)\omega(v)dv +\sum_{k=1}^\infty \int_{D(k)} h^{\tau,p}_{t,R}(v)\omega(v)dv.
\end{align}
\noindent As a consequence of log-homogeneity (\cite{MzM1}, Proposition $7.1$), the weight function $\omega(v)$ is bounded by a polynomial of uniformly bounded degree in powers of $\log\lVert v\rVert$, and hence by a polynomial in $r(v)$ by Lemma \ref{distancebound}. This will be used in the following.

The integral over $D(0)$ can be handled with the use of Proposition \ref{traceheatdecay}. Indeed, we get the bound
\begin{align*}
    \left\vert \int_{D(0)} h^{\tau,p}_{t,R}(v)\omega(v)dv\right\vert &\leq C e^{-\lambda t}\int_{D(0)} (1+r(v))^kdv \\
    &\leq Ce^{-\lambda t} (1+A)^k \vol(D(0)).
\end{align*}
As $D(0)$ is a compact domain, thus of finite volume, only depending on $A$ and $V$, and we consider only finitely many such groups $V$, we see that the above bound is of the form $C' e^{-\lambda t}$ for some constant $C'$ only depending on $G$. 

For the integral over $D(k)$, $k\geq 1$, we note that (\ref{vanishin|g|}) is applicable. Hence we get
\begin{align}\label{firststeparch}
    \left\vert \int_{D(k)} h^{\tau,p}_{t,R}(v)\omega(v)dv\right\vert &\leq C e^{-\lambda t}\int_{D(k)} e^{-c\frac{r(v)^2}{t}}(1+r(v))^bdv.
\end{align}
Pick any small $\epsilon>0$ and a constant $C''>0$ such that $(1+r(v))^b\leq C''e^{\epsilon r(v)}$. The volume of $D(k)$ is bounded by the volume of $V(\R)\cap B(A+k)$, which is again bounded by the volume of
\begin{align}\label{ballexp}
    \lbrace v\in V(\R) \mid \lVert v\rVert \leq e^{A+k}\rbrace
\end{align}
by Lemma \ref{distancebound}. since the Haar measure on $V(\R)$ is compatible with the Lebesgue measure on the Euclidean space $\R^{\dim V}$, the volume of the ball (\ref{ballexp}) is bounded by polynomial of degree $\dim V$ in the radius. Hence, for some constant $C_A$ depending on $A$, we get
\begin{align*}
    \vol(D(k)) \leq C_A e^{\dim V \cdot\, k}.
\end{align*}
Let $c_1 = \dim V+\epsilon$. Then we can bound the integral on the right hand side of (\ref{firststeparch}) by
\begin{align*}
    \int_{D(k)} e^{-c\frac{r(v)^2}{t}}(1+r(v))^kdv &\leq Ce^{-c\frac{(k+A-1)^2}{t}}e^{\epsilon k}\vol(D(k)) \\
    &\leq C_A' e^{-c\frac{k^2}{t}}e^{c_1 k},
\end{align*}
for some other constant $C'_A$ depending only on $A$. This bound guarantees the absolute convergence of the sum in (\ref{D(k)decomp}), meaning we can bound it by
\begin{align*}
    C_A'\,\sum_{k=1}^\infty e^{-c\frac{k^2}{t}}e^{c_1 k}.
\end{align*}
As the function $e^{-c\frac{x^2}{t}}e^{c_1 x}$ for $x\geq 0$ is increasing for $2c\frac{x}{t}-c_1\le 0$, and decreasing elsewhere, the above sum can be bounded by the corresponding integral up to an absolute constant, and this integral is computable:
\begin{align*}
    \int_0^\infty e^{-c\frac{x^2}{t}}e^{c_1 x} dx = \sqrt{\frac{\pi t}{4c}}e^{c_1^2t/c}\left(1-\text{erf}\left(c_1\sqrt{\frac{t}{c}}\right)\right).
\end{align*}
Here $\text{erf}$ is the \textit{error function} $\text{erf}(z) = \frac{\sqrt{\pi}}{2}\int _0^z e^{-t^2}dt$, satisfying $0< \text{erf}(z) <1$ for all $z>0$. Set $c' = \frac{c_1^2}{c}$. Putting everything together, we have proven the following result.

\begin{prop}\label{orbitasymp}
    Assume $t> 1$. Then there exist constants $C_{\tau,p}>0$ only depending on $\tau$ and $p$ and $c'>0$ only depending on $G$ such that
    \begin{align*}
        \left\vert\int_{V(\R)}h^{\tau,p}_{t,R}(v)\omega(v)dv\right\vert \leq C_{\tau,p} \,e^{-(\lambda-c') t}.
    \end{align*}
\end{prop}
\noindent This gives us the desired large $t$ asymptotics (\ref{largef}) if one assumes a large enough $\lambda$. We will put this to work in the following section.

\bigskip

\section{Proof of the main theorem}\label{conclusion}

\noindent We continue to assume $G=\GL(n)$, $n\geq 3$ and $\tau$ is $\lambda$-strongly acyclic. To prove Theorem \ref{glmytheorem}, we first turn to an analysis of the terms (\ref{mellin}). We truncate the integral in this expression with respect to a parameter $T>0$, leaving a remainder term.
\begin{align}\label{trunc}
    \int_0^\infty \Trreg(e^{-t\Delta_p(\tau)})t^{s-1}dt = &\int_0^T \Trreg(e^{-t\Delta_p(\tau)})t^{s-1}dt \\
    + &\int_T^\infty \Trreg(e^{-t\Delta_p(\tau)})t^{s-1}dt.
\end{align}
We focus first on the second integral on the right hand side, what we call the remainder. It is convergent for all $s$, in particular it is holomorphic at $s=0$. For any meromorphic function $f$ holomorphic at $0$ it holds that $\FP_{s=0}(f(s))=f(0)$, thus by linearity, we need only analyze
\begin{align*}
    E_0(0,T)\coloneqq \int_T^\infty \Trreg(e^{-t\Delta_p(\tau)})t^{-1}dt.
\end{align*}
By applying our large $t$ asymptotic for the spectral side (\ref{larget}) and its following remark, and using the trace formula, i.e. the equality of the geometric side and spectral side, we get the following lemma.

\begin{lem}\label{error0}
There exists a $C>0$ only depending on $G$ such that for any $\epsilon>0$,
    \begin{align*}
    \vert E_0(0,T)\vert\leq Ce^{-\lambda(1-\epsilon)T}\vol(Y(N)).
\end{align*}
\end{lem}

\medskip

\subsection*{Restricting to unipotent contribution}

\noindent We return to the first integral on the right-hand side of (\ref{trunc}). We want to substitute the test function with its compactified version. By Proposition \ref{compactdiff}, we may write
\begin{align*}
    \int_0^T J_{\geo}(h^{\tau,p}_t\otimes\chi_{K(N)})t^{s-1}dt = \int_0^T J_{\geo}(h^{\tau,p}_{t,R}\otimes\chi_{K(N)})t^{s-1}dt+E_1(s,R,T),
\end{align*}
with the error term $E_1(s,R,T)$ given by an integral convergent for all $s\in\C$ and satisfying
\begin{align}\label{error1}
    |E_1(0,R,T)| &\leq C_3\int_0^Te^{-C_1R^2/t+C_2t}t^{-1}dt\vol(Y(N)) \\
    &\leq C_3e^{-C_4R^2/T+C_2T}\int_0^{T/R^2}e^{-C_4/t}t^{-1}dt\vol(Y(N)).
\end{align}
Again, $C_2,C_4>0$ both only depend on the group $G$. We will return to this estimate in a moment. The reward for compactifying our test function is that by Proposition \ref{compactunip}, only the unipotent part of the coarse geometric expansion (\ref{coarsegeom}) contribute if we keep the radius of the support small enough. Assume $R\leq C_n\log N$. Then
\begin{align*}
    \int_0^T J_{\geo}(h^{\tau,p}_{t,R}\otimes\chi_{K(N)})t^{s-1}dt = \int_0^T J_{\unip}(h^{\tau,p}_{t,R}\otimes\chi_{K(N)})t^{s-1}dt
\end{align*}

\bigskip

\subsection*{Applying the fine geometric expansion}\label{applyingfine}

To deal with the truncated Mellin transform of $J_{\unip}$, we recall the fine geometric expansion (\ref{finegeom}). We isolate the term for $(M,\mf{u})=(G,\lbrace 1\rbrace)$, which is exactly $h^{\tau,p}_{t,R}(1)\vol(Y(N))$, and denote the remaining sum by $J_{\unip-1}$. This allows us to write
\begin{align}\label{unip-1}
    \int_0^T J_{\unip}(h^{\tau,p}_{t,R}\otimes\chi_{K(N)})t^{s-1}dt &= \int_0^T h_{t,R}^{\tau,p}(1)t^{s-1}dt \vol(Y(N)) \\
    &+\int_0^T J_{\unip-1}(h^{\tau,p}_{t,R}\otimes\chi_{K(N)})t^{s-1}dt .
\end{align}
The first integral on the right-hand side is the truncated Mellin transform of $h_{t,R}^{\tau,p}(1)$. As we are evaluating at $1\in G(\R)$, the compactification has no effect, and hence we can ignore $R$. By (\cite{MzM1}, ($5.11$)) we have an asymptotic expansion
\begin{align*}
    h_t^{\tau,p}(1)\sim \sum_{i=0}^\infty a_i t^{-d/2+i}, \quad t\to 0.
\end{align*}
Furthermore, as a special case of Theorem \ref{traceheatdecay} we have for $t\geq 1$,
\begin{align*}
    |h_t^{\tau,p}(1)|\leq C_{\tau,p}e^{-\lambda t}.
\end{align*}
Thus, the integral is convergent for $s$ in some half-plane and has a meromorphic extension to all of $\C$. We may write
\begin{align*}
    \int_0^T h_{t,R}^{\tau,p}(1)t^{s-1}dt \vol(Y(N)) = \int_0^\infty h_t^{\tau,p}(1)t^{s-1}dt \vol(Y(N)) + E_2(s,T),
\end{align*}
with the error term $E_2(s,T)$ given by an integral convergent for all $s\in\C$ and satisfying
\begin{align}\label{error2}
    |E_2(0,T)|&\leq C\int_T^\infty e^{-\lambda t}t^{-1}dt \vol(Y(N)) \\
    &\leq C'e^{-\lambda T}\vol(Y(N))
\end{align}
for $T\geq 1$. We return to the second integral on the right-hand side of (\ref{unip-1}). Using the splitting of the orbital integrals into local parts, i.e. equation (\ref{localparts}), we may express this integral as
\begin{align}\label{expansioncalc}
    \sum_{(M,\mf{u})\neq(G,\lbrace 1\rbrace)}a^M(S(N),\mf{u})\sum_{L_1,L_2\in \mathcal{L}(M)}d_M(L_1,L_2)A_M^{L_1}(h^{\tau,p}_{t,R},\mf{u}_\infty,T)J_M^{L_2}(\chi_{K(N)},\mf{u}_{\text{fin}}),
\end{align}
with the archimedean part defined as 
\begin{align*}
    A_M^{L}(h^{\tau,p}_{t,R},\mf{u}_\infty,T) = \int_0^T J^L_M((h^{\tau,p}_{t,R})_Q,\mf{u}_\infty)t^{s-1}dt.
\end{align*}
We will treat this term like we did the integral of $h^{\tau,p}_{t,R}(1)$ above. By the asymptotics given in Proposition \ref{orbitsmallt}, this is convergent for $s$ in some half plane. By Proposition \ref{orbitasymp}, we may write
\begin{align}\label{archterm}
    A_M^{L}(h^{\tau,p}_{t,R},\mf{u}_\infty,T) = \int_0^\infty J^L_M((h^{\tau,p}_{t,R})_Q,\mf{u}_\infty)t^{s-1}dt+E_3(s,T),
\end{align}
with $E_3(s,T)$ some error term, given by an integral convergent for all $s\in\C$, analogously to $E_2$ satisfying
\begin{align}\label{error3}
    |E_3(0,T)|\leq C'' e^{-(\lambda-c') T} 
\end{align}
for $T> 1$ and some constants $c',C''>0$. We will write $A_M^{L}(h^{\tau,p}_{t,R},\mf{u}_\infty)$ for the integral on the right-hand side of (\ref{archterm}), which is convergent in some half plane with a meromorphic extension to all of $s\in\C$. This means that
\begin{align*}
    \text{FP}_{s=0}\left(\frac{1}{s\Gamma(s)}A_M^{L}(h^{\tau,p}_{t,R},\mf{u}_\infty)\right)
\end{align*}
is well defined. 
\begin{lem}\label{unip-1bound}
    Assume $\lambda > c_1$. Then there exists constants $C>0$ only depending on $\tau,p$ and $\lambda$ and $b>0$ only depending on the group such that for $T> 1$, we have
    \begin{align*}
    &\Bigg\vert\FP_{s=0}\left(\frac{1}{s\Gamma(s)}\int_0^T J_{\unip-1}(h^{\tau,p}_{t,R}\otimes \chi_{K(N)})t^{s-1}dt\right)\Bigg\vert  \\
    &\leq\:\: C N^{-(n-1)}(1+\log N)^b \vol(Y(N)).
\end{align*}
\end{lem}

\begin{proof}
    Considering the expression (\ref{expansioncalc}) and the discussion above, it suffices to appropriately bound the three following terms:
    \begin{align*}
        &a^M(S(N),\mf{u}), \\
        & J_M^{L_2}(\chi_{K(N)},\mf{u}_{\text{fin}}) \:\: \text{for} \:\:(M,\mf{u})\neq (G,\lbrace1\rbrace), \\
        &\FP_{s=0}\left(\frac{1}{s\Gamma(s)}\left(A_M^{L}(h^{\tau,p}_{t,R},\mf{u}_\infty)+E_3(s,T)\right)\right).
    \end{align*}

    \noindent The global coefficients $a^M(S(N),\mf{u})$ were bounded in Lemma \ref{coeffbound}, and the local orbital integrals $J_M^{L_2}(\chi_{K(N)},\mf{u}_{\text{fin}})$ were bounded in Lemma \ref{localbound} and its remark. Note that $\vol(Y(N))=c\vol(K(N))^{-1}$ for some constant $c$ depending only on normalisations of measures (see the appendix).

    By Proposition \ref{orbitsmallt} and its remark, alongside Proposition \ref{orbitasymp}, the term 
    $$\FP_{s=0}\left(\frac{1}{s\Gamma(s)}A_M^{L}(h^{\tau,p}_{t,R},\mf{u}_\infty)\right)$$
    is uniformly bounded by a constant, following the theory of Mellin transforms. The bound can be chosen to be uniform over all $(L,M,\mf{u}_\infty)$, as there are only finitely many such triples. Finally, the term
    $$\FP_{s=0}\left(\frac{1}{s\Gamma(s)}E_3(s,T)\right) = E_3(0,T)$$ 
    was bounded in (\ref{error3}), and we see that it is vanishing in $T$ if $\lambda >c_1$, in particular bounded by a constant. Collecting all these bounds and applying to (\ref{expansioncalc}), we arrive at the desired bound. 
\end{proof}

\subsection*{Asymptotics of analytic torsion}

Recall the definition of analytic torsion (\ref{torsion}). By collecting our initial analysis in the previous subsections, we may write

\begin{align}\label{analbound}
    \log T_{Y(N)}(\tau) &= \FP_{s=0}\left(\frac12\frac{1}{s\Gamma(s)}\int_0^\infty\sum_{p=1}^d (-1)^p p \, h^{\tau,p}_t(1)t^{s-1}dt\right) \vol(Y(N)) \\\label{FPunip-1}
    &+ \FP_{s=0}\left(\frac12\sum_{p=1}^d (-1)^p p\,\frac{1}{s\Gamma(s)} \int_0^T J_{\unip-1}(h^{\tau,p}_{t,R}\otimes\chi_{K(N)})t^{s-1}dt\right) \\
    &+ \frac12\sum_{p=1}^d (-1)^p p\, \left(E_0(0,T) + E_1(0,R,T)+E_2(0,T)\right)
\end{align}
The first term on the right-hand side is the $L^2$-torsion of $Y(N)$ (see \cite{Lott}). We set
\begin{align*}
    t^{(2)}_{\Tilde{X}}(\tau)\coloneqq  \frac12\frac{d}{ds}\left( \frac{1}{\Gamma(s)}\int_0^\infty\sum_{p=1}^d (-1)^p p \, h^{\tau,p}_t(1)t^{s-1}dt\right)\Bigg\vert_{s=0}.
\end{align*}
Note that although $\frac{1}{\Gamma(s)}\int_0^\infty h^{\tau,p}_t(1)t^{s-1}dt$ may not be holomorphic at $s=0$, swapping $h^{\tau,p}_t(1)$ with this alternating sum over $p$ seen above turns the meromorphic extension holomorphic at $s=0$, from which the above definition is well defined (see \cite{BV}, §$4.4$). Then the first term is exactly $\log T^{(2)}_{Y(N)} = t^{(2)}_{\Tilde{X}}(\tau)\vol(Y(N))$. All that is left to prove Theorem \ref{glmytheorem} is to show that the second and third term above both have the form $O(\vol(Y(N))N^{-(n-1)}\log(N)^b)$.

This was accomplished for the second term in Lemma \ref{unip-1bound} when $\lambda>c_1$, as we assume $R\leq C_n\log N$. Note that the sum over $p$ contributes at most a constant multiple, as the bound was independent of $p$. To arrive at our desired bound for the third and final term, we wish to pick our parameters $R$ and $T$ such that the error terms $E_0,E_1,E_2$ are all as small as possible. This is done in the following section.

\subsection*{A dance of error terms}\label{dance}

We pick $T=\beta \log N$ and $R=C_n\log N$, with $C_n$ from Proposition \ref{compactunip} and $\beta>0$ still to be determined. We see that the integral $\int_0^{T/R^2}e^{-C_4/t}t^{s-1}dt$ is then vanishing in $N$, in particular bounded by a constant for $N\geq 3$. Thus, by Lemma \ref{error0}, (\ref{error1}), and (\ref{error2}) respectively we have
\begin{align*}
    E_0(0,\beta\log N) &= O( N^{-\lambda(1-\epsilon)\beta}\vol(Y(N))),\\
    E_1(0,C_n \log N,\beta\log N) &= O( N^{-C_4C_n^2/\beta +C_2\beta}\vol(Y(N))),\\
    E_2(0,\beta\log N) &= O( N^{-\lambda\beta}\vol(Y(N))).
\end{align*}
The implied constants may depend on $\tau$. We now ensure that the error terms above are of the size
\begin{align}\label{secondordersize}
    O(\vol(Y(N)N^{-(n-1)}(\log N)^a)
\end{align}
for some $a>0$ and $N$ large enough. Here is where the independence of the constants $C_2,C_4$ on $\tau$ becomes important. 

Pick $\beta$ such that 
$$C_4C_n/\beta-C_2\beta>n-1.$$
This choice of $\beta$ then only depends on $n$, i.e. on the group $G$. With $\beta$ fixed, let $\lambda$ be chosen such that 
\begin{align*}
    \lambda(1-\epsilon)\beta&>n-1, \:\text{ and}\\
    \lambda &> c'.
\end{align*}

\noindent The latter inequality is to ensure we can use Lemma \ref{unip-1bound}. As $\epsilon$ can be chosen arbitrarily close to $0$, for the first inequality it is in fact sufficient to satisfy $\lambda>\frac{n-1}{\beta}$. Note that this choice of $\lambda$ also only depends on $C_2$ and $C_4$, which only depend on $G$. This implies that for $\tau$ being $\lambda$-strongly acyclic for the any $\lambda$ satisfying this inequality, all the error terms are of the form (\ref{secondordersize}) as desired.

We have now proven Theorem \ref{glmytheorem}, from which Theorem \ref{mytheorem} follows.

\clearpage
\bigskip

\pagebreak

\section{Appendix}
\noindent For the convenience of the reader, in this section we prove the well known formula of the volume of the group $K(N)$ and show its inverse relation to the volume of its associated adelic locally symmetric space. This is summarized in the following proposition.

\begin{prop}
    Let $K(N)$ be the open compact subgroup of $\GL(n,\A_f)$ given by $K(N) = \prod_p K_p(p^{v_p(N)})$, where $K_p(p^k)=\ker(\GL(n,\Z_p)\to \GL(n,\Z/p^k\Z))$, and let $Y(N)$ be the associated adelic locally symmetric space
    \begin{align*}
        Y(N) = \GL(n,\Q)\backslash (\GL(n,\A)/\SO(n)K(N)).
    \end{align*} 
    Similarly, let $K'(N)$ be the analogous open compact subgroup of $\SL(n,\A_f)$, with adelic locally symmetric space $X(N)$. Then the following formulas hold:
    \begin{align}
        \label{volformula}\vol(K(N)) &= \left(N^{n^2}\prod_{p\mid N}\prod_{k=1}^{n}\left(1-\frac{1}{p^k}\right)\right)^{-1}\\
        \label{relatevol}\vol(K(N)) &= \phi(N)^{-1}\vol(K'(N)).
    \end{align}
    Here $\phi$ is Euler's totient function. Furthermore, let
    $$c(n)\coloneqq \vol(\SL(n,\Z)\backslash(\SL(n,\R)/\SO(n)))^{-1}.$$
    This constant depends only on normalizations of measures. We have the relations
    \begin{align}
        \label{GLnspacevol}\vol(K(N)) &= c(n) \vol(Y(N))^{-1}\\
        \label{SLnspacevol}\vol(K'(N)) &= c(n) \vol(X(N))^{-1}.
    \end{align}
\end{prop}

\begin{proof}
As $K(N) = \prod_p K_p(p^{v_p(N)})$, and the measures respect this decomposition, it is sufficient to compute $\vol(K_p(p^{v_p(N)}))$. Fix $p$, and let $m=v_p(N)$. $K_p(p^m)$ is the kernel of the surjective map $\GL(n,\Z_p)\to\GL(n,\Z_p/p^m\Z_p)\cong \GL(n,\Z/p^m\Z)$. Thus,
$$\GL(n,\Z_p)=\bigsqcup_{a\in\GL(n,\Z/p^m\Z)}(a\cdot K_p(p^m))$$
where we by abuse of notation let $a\in\GL(n,\Z/p^m\Z)$ also denote any lift to $\GL(n,\Z_p)$. Since the cosets are disjoint, and they are homeomorphic through multiplication, they all have the same measure, $\vol(K_p(p^m))$. Since the left-hand side has measure $1$ in our normalization, we get that 
$$\vol(K_p(p^m))=\frac{1}{|\GL(n,\Z/p^m\Z)|}.$$

\noindent To determine $|\GL(n,\Z/p^m\Z)|$, we use the following short exact sequence of finite groups
$$0\to\lbrace I+pA\mid A\in M_n(\Z/p^m\Z)\rbrace\to \GL(n,\Z/p^m\Z)\to \GL(n,\Z/p\Z)\to 0.$$

\noindent The rightmost group has cardinality $\prod_{k=0}^{n-1}(p^n-p^k)$ by counting number of bases, using that $\Z/p\Z$ is a field. The leftmost group is bijective to $\lbrace pA\mid A\in M_n(\Z/p^m\Z)\rbrace$, which is the kernel of $M_n(\Z/p^m\Z)\to M_n(\Z/p\Z)$, induced by the obvious quotient map. These have cardinality $p^{n^2m}$ and $p^{n^2}$ respectively, so their kernel has cardinality $p^{m^2(n-1)}$. Plugging everything in, we get that
\begin{align}\label{localvol}
 |\GL(n,\Z/p^m\Z)|\:&= |\GL(n,\Z/p\Z)|\cdot |\lbrace I+pA\mid A\in M_n(\Z/p^m\Z)\rbrace| \\
&= p^{n^2(m-1)}\prod_{k=0}^{n-1}(p^n-p^k) = (p^{m})^{n^2}\prod_{k=0}^{n-1}\left(1-\frac{1}{p^{n-k}}\right).
\end{align}
Reindexing the product, this establishes formula (\ref{volformula}). For the next formula, running through the same argument for $K'(N)$, we get that
\begin{align}\label{SLnlocalvol}
    \vol(K_p'(p^{m})) = \frac{1}{|\SL(n,\Z/p^m\Z)|}.
\end{align}
As the number of units of $\Z/p^m\Z$ is exactly $\phi(p^m)$, we get that 
\begin{align*}
    |\GL(n,\Z/p^m\Z)| = \phi(p^m) |\SL(n,\Z/p^m\Z)|,
\end{align*}
and this proves (\ref{relatevol}). 

Now, let us relate $\vol(Y(N))$ to $\vol(K(N))$. By combining (\ref{volglsl}) and (\ref{relatevol}), we see that the two latter statements of the proposition are equivalent, and so it is sufficient to prove (\ref{SLnspacevol}). By strong approximation, $X(N) =\Gamma(N)\backslash\SL(n,\R)/\SO(n)$, with $\Gamma(N)$ the classical principal congruence subgroup of level $N$. The fundamental domain of $X(N)$ can be viewed as $[\SL(n,\Z):\Gamma(N)]$ copies of a fundamental domain of $\SL(n,\Z)\backslash \SL(n,\R)/\SO(n)$, and hence
\begin{align*}
    \vol(X(N)) = \vol(\SL(n,\Z)\backslash \SL(n,\R)/\SO(n))[\SL(n,\Z):\Gamma(N)].
\end{align*}
Note that $[\SL(n,\Z):\Gamma(N)] = |\SL(n,\Z/N\Z)|$ essentially by definition of $\Gamma(N)$. By the Chinese remainder theorem and (\ref{SLnlocalvol}), we have that
\begin{align*}
    |\SL(n,\Z/N\Z)| = \prod_{p\mid N}|\SL(n,\Z/p^{v_p(N)}\Z)| = \vol(K'(N))^{-1}.
\end{align*}
This concludes the proof.
\end{proof}

\noindent The proposition immediately implies $\vol(X(N)) = O(N^{n^2-1})$. In particular, we get the inequality
\begin{align*}
    N^{-(n-1)}= O\left(\vol(X(N))^{-\frac{1}{n+1}}\right),
\end{align*}
with which we can rewrite the main theorem into its invariant form in Remark \ref{mytheoreminvar}.

\bigskip

\pagebreak

\bibliographystyle{amsalpha} 
\bibliography{refs} 

\end{document}